\documentclass{amsart}

\usepackage[all]{xy}
\usepackage{amsmath}
\usepackage{hyperref}
\usepackage{amsfonts,graphics,amsthm,amsfonts,amscd,latexsym}
\usepackage{epsfig}
\usepackage{flafter}
\usepackage{mathtools}
\usepackage{comment}
\usepackage{stmaryrd}
\usepackage{tabularx}
\usepackage{pst-node}

\usepackage{tikz-cd}
\hypersetup{
    colorlinks=true,    
    linkcolor=blue,          
    citecolor=blue,      
    filecolor=blue,      
    urlcolor=blue           
}

\usepackage{pgfplots}
\pgfplotsset{compat=1.15}
\usepackage{mathrsfs}
\usepackage{tikz}
\usetikzlibrary{graphs,positioning,arrows,shapes.misc,decorations.pathmorphing}

\tikzset{
    >=stealth,
    every picture/.style={thick},
    graphs/every graph/.style={empty nodes},
}

\tikzstyle{vertex}=[
    draw,
    circle,
    fill=black,
    inner sep=1pt,
    minimum width=5pt,
]
\usepackage[most]{tcolorbox}

\tcbset{
  question style/.style={
    colback=gray!8,        
    colframe=black!60!black, 
    boxrule=0.8pt,
    arc=2mm,               
    left=3mm, right=3mm,   
    top=1mm, bottom=1mm,
    breakable,             
    enhanced
  }
}

\newtcolorbox{questionbox}[1][]{%
  question style,
  title=\textbf{Question},  
  #1                        
}

\usepackage[position=top]{subfig}
\usepackage{amssymb}
\usepackage{color}

\calclayout

\usetikzlibrary{decorations.pathmorphing}
\tikzstyle{printersafe}=[decoration={snake,amplitude=0pt}]

\newcommand{\Aut}{\operatorname{Aut}}

\newcommand{\Cl}{\operatorname{Cl}}

\newcommand{\Spec}{\operatorname{Spec}}

\usepackage{tikz}
\usetikzlibrary{matrix,arrows,decorations.pathmorphing}
\usetikzlibrary{arrows}

\definecolor{uuuuuu}{rgb}{0.26666666666666666,0.26666666666666666,0.26666666666666666}

  \newtheorem{introthm}{Theorem}

  \newtheorem{theorem}{Theorem}[section]
  \newtheorem{lemma}[theorem]{Lemma}
  \newtheorem{proposition}[theorem]{Proposition}
  \newtheorem{corollary}[theorem]{Corollary}

  \newtheorem{theorem/definition}[theorem]{Theorem/Definition}

  \newtheorem{definition}[theorem]{Definition}
  \newtheorem{example}[theorem]{Example}

\theoremstyle{remark}

\numberwithin{equation}{section}



\begin{document}

\title{Geometry of Tropical Mutation Surfaces with a Single Mutation}

\author[T.~Oda]{Tomoki Oda}
\address{UCLA Mathematics Department, Box 951555, Los Angeles, CA 90095-1555, USA
}
\email{tomokioda0723@math.ucla.edu}
\thanks{The author was partially supported by NSF research grant DMS-2443425.}

\begin{abstract}

Escobar, Harada, and Manon introduced polyptych lattices as a
piecewise-linear extension of the lattice-polytope formalism of toric
geometry. In this paper we study the first genuinely non-toric case:
rank-two polyptych lattices with a single shear. A detropicalization is
given by a polynomial \(f(y)\), and the corresponding affine surface is
 $U_f=\operatorname{Spec} K[x_1,x_2,y^{\pm 1}]/\langle x_1x_2-f(y)\rangle.$
We classify these detropicalizations, compute the complexity of their
projective compactifications, and show that the resulting log
Calabi--Yau surface pairs are of cluster type. Conversely, we prove that every normal projective \(\mathbb Q\)-factorial
index-one log Calabi--Yau surface pair with reduced boundary, ample boundary
support, and a nontrivial \(\mathbb G_m\)-action arises from this single-shear
construction. We also construct a global family
interpolating between the two toric degenerations associated with the two
charts, and compute the Cox rings of the resulting tropical mutation
surfaces.
\end{abstract}

\maketitle

\setcounter{tocdepth}{1}
\tableofcontents
\section{Introduction}

Toric geometry translates the geometry of algebraic varieties with a dense
torus into the combinatorics of lattices and polytopes. In particular, the
Laurent polynomial ring of a lattice plays the role of the coordinate ring of
the algebraic torus, and lattice polytopes determine projective toric
compactifications. In~\cite{EHM24}, Escobar, Harada, and Manon introduced
\emph{polyptych lattices} as a framework for extending this correspondence
beyond the toric setting. A polyptych lattice is a collection of lattices,
called charts, glued by piecewise-linear identifications called mutations
\cite{Akhtar_2015,KAN}. To such a polyptych lattice $\mathcal M$, they
associate a commutative algebra $A_{\mathcal M}$, called a detropicalized
algebra. Its spectrum $U_{A_{\mathcal M}}:=\Spec A_{\mathcal M}$
is called the affine tropical mutation variety (cf.~Definition~\ref{def:detro}).

The goal of this paper is to study the first genuinely non-toric case in
detail. We focus on the rank-two shearing polyptych lattices $\mathcal M_s$,
introduced in~\cite{CEHM24}. Each $\mathcal M_s$ is obtained by gluing two
rank-two lattices by a shear across two linear regions. This is arguably the
simplest polyptych lattice with a nontrivial mutation, and therefore provides a
natural testing ground for the geometry of tropical mutation varieties.

In this case, the detropicalized algebra is completely explicit. A
detropicalization of $\mathcal M_s$ is determined by a monic polynomial
$f(y)=y^s+b_{s-1}y^{s-1}+\cdots+b_1y+1
        \in \mathbb K[y],$
and the corresponding affine tropical mutation surface is $U_f
        =
        \Spec \mathbb K[x_1,x_2,y^{\pm 1}]/\langle x_1x_2-f(y)\rangle .$
Thus the rank-two shearing case replaces the algebraic torus
$\mathbb G_m^2$ by the explicit log Calabi--Yau surface $x_1x_2=f(y)$.
Different choices of $f$ give nontrivial deformations of the affine surface,
a phenomenon absent in the toric case. 
Our first result classifies these detropicalizations.

\begin{introthm}[cf.~Theorem~\ref{theorem:moduli}]\label{thm:intro_moduli}
Let \(s\geq 1\), and let \(\zeta\) be a primitive \(s\)-th root of unity. Let
\(\mu_s=\langle \zeta\rangle\subset \mathbb K^\times\), and let $D_{2s}:=\mu_s\rtimes\langle \tau\rangle$ be the dihedral group. It acts on $\mathbb A^{s-1}=\Spec \mathbb K[b_1,\ldots,b_{s-1}]$
by
\[
        \zeta\cdot(b_1,\ldots,b_{s-1})
        =
        (\zeta b_1,\zeta^2 b_2,\ldots,\zeta^{s-1}b_{s-1})\qquad\text{and}\qquad
        \tau\cdot(b_1,\ldots,b_{s-1})
        =
        (b_{s-1},b_{s-2},\ldots,b_1).
\]
Then the coarse moduli space of detropicalizations of \(\mathcal M_s\) is $\mathbb A^{s-1}/D_{2s}.$
\end{introthm}

Escobar, Harada, and Manon also associate compactifications to polyptych
polytopes. Given a polytope \(\mathcal P\subset \mathcal M\), they construct a
projective compactification of \(U_{A_{\mathcal M}}\) by adding a boundary
divisor \(B(\mathcal P)\), which we denote by $X_{A_{\mathcal M}}(\mathcal P).$
When no mutation is present, this recovers the usual toric compactification
associated to a lattice polytope. In the shearing case, we write $(X_f(\mathcal P),B(\mathcal P))$
for the resulting tropical mutation surface pair.

A guiding question is how close these pairs are to toric pairs. For a toric
surface pair \((T,B_T)\), the boundary has complexity zero in the sense of
Brown--McKernan--Svaldi--Zong~\cite{BMSZ18}. We show that tropical mutation
surface pairs have a similarly rigid complexity theory: their complexity is
independent of the compactifying polytope and is controlled only by the
factorization type of \(f\).

\begin{introthm}[cf.~Corollary~\ref{cor:complexity}]\label{cluster Theorem}
\label{thm:intro_complexity}
Let \(\mathcal P\) be a polytope in the rank-two shearing polyptych lattice
\(\mathcal M_s\), and let $(X_f(\mathcal P),B(\mathcal P))$
be the associated tropical mutation surface pair. Then:
\begin{enumerate}
    \item \(B(\mathcal P)\) supports an effective ample divisor;
    \item \(\mathbb G_m\leq \Aut(X_f(\mathcal P),B(\mathcal P))\);
    \item \((X_f(\mathcal P),B(\mathcal P))\) is of cluster type; and
    \item the complexity of \((X_f(\mathcal P),B(\mathcal P))\) is equal to
    the number of distinct roots of \(f\).
\end{enumerate}
\end{introthm}

The previous theorem shows that tropical mutation surface pairs form a natural
class of log Calabi--Yau \(\mathbb G_m\)-surfaces. Conversely, we prove that
this class is intrinsic. Namely, in dimension two, the shearing tropical
mutation construction is not merely a source of examples: it characterizes all
\(\mathbb Q\)-factorial index-one log Calabi--Yau surface pairs with ample
boundary support and a nontrivial \(\mathbb G_m\)-action.

\begin{introthm}[cf.~Theorem~\ref{thm:cluster-tropical}]\label{thm:intro_characterization}
Let \((X,B)\) be an index-one normal projective \(\mathbb Q\)-factorial
log Calabi--Yau surface pair with reduced boundary. Assume that \(B\)
supports an effective ample divisor and that \(\mathbb G_m\leq \Aut(X,B)\)
acts nontrivially. Then \((X,B)\) is a shearing tropical mutation surface pair.
\end{introthm}

The proof proceeds by first showing that such a pair is of cluster type. Then,
using the \(\mathbb G_m\)-action, we show that its cluster presentation may be
chosen so that all non-toric boundary centers lie on a single boundary
component after equivariant elementary transformations. This reduces the pair
to the inverse of an \(s\)-collinear blow-up of a toric surface pair, which is
precisely the surface operation encoded by the rank-two shearing mutation.

We next study toric degenerations. For a polytope
\(\mathcal P\subset\mathcal M_s\), let \(P_1\) and \(P_2\) denote its images in
the two toric charts of \(\mathcal M_s\), and let \(T(P_i)\) be the associated
projective toric surfaces. The tropical mutation surface \(X_f(\mathcal P)\)
degenerates to both toric surfaces. We construct a single global family
interpolating between these two toric degenerations.

\begin{introthm}[cf.~Theorems~\ref{toricdegene} and~\ref{coincidewith}]
\label{thm:intro_global_mutation}
Let \(s\geq 1\), and let \(f\in \mathbb K[y]\) be a monic polynomial of degree
\(s\) with constant term \(1\). Let \(\mathcal P\subset\mathcal M_s\) be a
polytope, and let \(P_1\) and \(P_2\) be its two chart images. Then there
exists a projective flat family
\[
        \pi_{f,\mathcal P}\colon
        \mathcal X_f(\mathcal P)\longrightarrow \mathbb P^1\qquad \text{such that}
 \qquad     \mathcal X_{f,0}(\mathcal P)\cong T(P_1),
        \qquad
        \mathcal X_{f,\infty}(\mathcal P)\cong T(P_2),
\]
\[
\text{and}\qquad         \mathcal X_f(\mathcal P)|_{\mathbb P^1\setminus\{0,\infty\}}
        \cong
        X_f(\mathcal P)\times(\mathbb P^1\setminus\{0,\infty\}).
\]
Moreover, the total space admits a complexity-one \(T\)-variety description:
there is a divisorial fan \(\mathcal S_f\) on
\(\mathbb P^1\times\mathbb P^1\) such that $\mathcal X_f(\mathcal P)\cong X(\mathcal S_f),$
and the morphism \(\pi_{f,\mathcal P}\) is induced by projection to the second
factor.
\end{introthm}

Finally, we compute the Cox rings of the tropical mutation surfaces
\(X_f(\mathcal P)\). We first show that these surfaces are Mori dream spaces.
Then, using the description of complexity-one \(T\)-varieties by polyhedral
divisors due to Hausen and Süß~\cite{HS10}, we obtain explicit generators and
relations. The result expresses the Cox ring entirely in terms of the
combinatorial data of the tropical mutation boundary and the roots of \(f\).

\begin{introthm}[cf.~Theorem~\ref{thm:Cox_ring_presentation}, Corollary~\ref{criteria}]
\label{thm:intro_cox}
Let $f(y)=\prod_{k=1}^{\gamma}(y-\alpha_k)^{\beta_k}$
be a degree \(s\) polynomial with \(f(0)=1\), where
\(\alpha_k\in\mathbb K^\times\) and \(\sum_{k=1}^{\gamma}\beta_k=s\).
Let \(\mathcal P\subset\mathcal M_s\) be a polytope defined by tropical
points $p_i=(a_i,b_i,c_i),\qquad i=1,\ldots,n,$
ordered so that $c_1,\ldots,c_j>0,\,
        c_{j+1},\dots,c_m<0,$
and the remaining \(c_i\)'s are zero. Let \(D_i\) be the irreducible boundary
divisor corresponding to \(p_i\), and let $C_1,\ldots,C_{2\gamma}$
be the interior divisors arising from the distinct roots of \(f\). Then:

\begin{enumerate}
\item The class group of \(X_f(\mathcal P)\) is
\[
\Cl(X_f(\mathcal P))
\cong
\frac{
\mathbb Z\langle D_1,\ldots,D_n,C_1,\ldots,C_{2\gamma}\rangle
}{
\left\langle
\sum_{i=1}^n c_iD_i,\;
\sum_{i=1}^n a_iD_i+\sum_{k=1}^{\gamma}\beta_k C_{2k},\;
-\sum_{i=1}^j c_iD_i+C_{2k-1}+C_{2k}
\;\middle|\;
k=1,\ldots,\gamma
\right\rangle
}.
\]

\item The Cox ring of \(X_f(\mathcal P)\) is
\[
\operatorname{Cox}(X_f(\mathcal P))
\cong
\frac{
\mathbb K[w_1,\ldots,w_{n+2\gamma}]
}{
\left\langle
w_{n+2k-1}w_{n+2k}
+\alpha_k\prod_{i=1}^j w_i^{c_i}
-\prod_{i=j+1}^m w_i^{-c_i}
\;\middle|\;
k=1,\ldots,\gamma
\right\rangle
}.
\]
Here \(w_i\) is the canonical section of \(D_i\) for \(1\leq i\leq n\), and
\(w_{n+k}\) is the canonical section of \(C_k\) for
\(1\leq k\leq 2\gamma\).

\item The surface \(X_f(\mathcal P)\) is toric if and only if one of the
following combinatorial conditions holds:
\begin{enumerate}
    \item \(f\) has exactly one distinct root, and either exactly one
    \(c_i=1\) with all other \(c\)-weights nonpositive, or exactly one
    \(c_i=-1\) with all other \(c\)-weights nonnegative;
    \item \(f\) has exactly two distinct roots, exactly one \(c_i=1\),
    exactly one \(c_k=-1\), and all other \(c\)-weights are zero.
\end{enumerate}
\end{enumerate}
\end{introthm}

Together, these results show that the rank-two shearing polyptych lattices
provide a complete and computable model for the simplest non-toric tropical
mutation surfaces: their detropicalizations have an explicit moduli space,
their compactifications are intrinsically characterized among log
Calabi--Yau \(\mathbb G_m\)-surfaces, their toric degenerations fit into a
single global family, and their Cox rings admit concrete complete intersection
presentations.
\medskip

\textbf{Acknowledgments.} 
The author is grateful to Joshua Enwright, Laura Escobar, Megumi Harada, Nathan Ilten, Christopher Manon, Joaquín Moraga, and Burt Totaro for many helpful comments and suggestions.
The author is especially grateful to Joaquín Moraga for his invaluable guidance and supervision throughout this project.
The author thanks Nathan Ilten for many comments that helped to improve the content of this paper.

\section{Preliminaries}\label{sec: pre}

Throughout the article, we assume that $\mathbb{K}$ is an algebraically closed field of characteristic zero. We use the term polyptych lattice to mean a finite polyptych lattice. We adopt the framework of polyptych lattices and detropicalized algebras developed in~\cite{EHM24}.

\begin{definition}\label{def:detro}
{\rm (cf.~\cite[Definition~6.3]{EHM24})}
\emph{Let $\mathcal{M}$ be a polyptych lattice. Let 
$A_{\mathcal{M}}$ be a finitely generated detropicalized algebra.
The \emph{affine tropical mutation variety} is defined as
$U_{A_\mathcal{M}} := \operatorname{Spec}(A_{\mathcal{M}})$.}
\end{definition}

Throughout this article, we consider finitely generated $\mathbb{K}$-detropicalized algebras, as opposed to the more general Noetherian $\mathbb{K}$-algebras considered in \cite{EHM24}. For a fixed polyptych lattice, the choice of detropicalized algebra is not unique.

A detropicalized algebra associated with a strictly dualizable polyptych lattice (cf.~\cite[Definition~4.1]{EHM24}) admits a natural $\mathbb{K}$-vector space basis, called an adapted basis (cf.~\cite[Definition~6.5]{EHM24}). In this article, we denote elements of the adapted basis by $\theta$, using a notation that differs slightly from that of the original.

We use the term polytope in place of what is referred to as a 
PL polytope (cf.~\cite[Definition~5.1]{EHM24}) in the original literature. All polytopes will be 
assumed to be convex and integral and containing the origin.

\begin{definition}\label{Tropical}
{\rm (cf.~\cite[Theorem~7.6]{EHM24})}
\emph{The \emph{polytope algebra} is the graded algebra $A^{\mathcal{P}}_{\mathcal{M}} := \bigoplus_{k=0}^{\infty} \Gamma(U_{A_\mathcal{M}}, k\mathcal{P}) \cdot t^k.$
The associated projective variety
$X_{A_\mathcal{M}}(\mathcal{P}) := \operatorname{Proj} A^{\mathcal{P}}_{\mathcal{M}}$
is called the \emph{tropical mutation variety associated with $\mathcal{P}$}.}
\end{definition}

\begin{definition}\label{degenerable}
{\rm (cf.~\cite[Theorem~7.11]{EHM24})}
\emph{Let $\mathcal{M}$ be a strictly dualizable polyptych lattice and 
$(\mathcal{M}, \mathcal{N}, \operatorname{v}, \operatorname{w})$ a strictly dualizable pair. 
$\mathcal{M}$ is \emph{degenerable} if it is detropicalizable 
and there exists a strictly dual pair such that, for every $n_i \in \mathcal{N}$, 
there is a chart $\alpha \in I$ with $\pi_{\alpha}(\operatorname{v}(n_i))$ linear on $M_{\alpha}$.}
\end{definition}

For clarity of exposition, we use the term \emph{tropical point} in place of what is referred to as a "point" in the original literature (cf.~\cite[Definition~3.1]{EHM24}). Now we establish a correspondence between polytopes and the boundary divisors of tropical mutation varieties.

\begin{definition}
\emph{Let $\mathcal{M}$ be a degenerable polyptych lattice, and let $\mathcal{P} \subset \mathcal{M}$ be a polytope bounded by inequalities of the form 
$p_i - \alpha_i$ for tropical points $p_i \in \operatorname{Sp}(\mathcal{M})$ and integers $\alpha_i \in \mathbb{Z}_{<0}$. }    
\end{definition}

\begin{theorem}\label{Ample}
{\rm (cf.~\cite[Theorem~7.11]{EHM24})}
Let $X_{A_\mathcal{M}}(\mathcal{P})$ be the tropical mutation variety associated with $\mathcal{P}$, 
and let $U_{A_\mathcal{M}}$ denote the corresponding affine tropical mutation variety. Then:
\begin{enumerate}
    \item $U_{A_\mathcal{M}}$ is a dense open subvariety of $X_{A_\mathcal{M}}(\mathcal{P})$; 
    \item the complement $B(\mathcal{P}):=X_{A_{\mathcal{M}}}(\mathcal{P})\setminus U_{A_{\mathcal{M}}}$ is a divisor;
    \item for each irreducible component $D_i \in \operatorname{Supp}(B(\mathcal{P}))$, there exists a facet $\mathcal{F}_i$ of $\mathcal{P}$ such that $D_i$ is the vanishing locus of the adapted basis away from the facet $\mathcal{F}_i$; and
    \item $B(\mathcal{P})$ supports an effective ample divisor.
\end{enumerate}
\end{theorem}

\begin{definition}\label{def:trop_mut_boundary}
\emph{Let $\mathcal{P} \subset \mathcal{M}$ be a polytope. 
We define the divisor $B(\mathcal{P}) := X_{A_{\mathcal{M}}}(\mathcal{P}) \setminus U_{A_{\mathcal{M}}}$
and refer to it as the \emph{tropical mutation boundary}. 
The pair $\bigl(X_{A_{\mathcal{M}}}(\mathcal{P}), B(\mathcal{P})\bigr)$ 
is called the \emph{tropical mutation pair}. 
For each facet $\mathcal{F}_i$ of $\mathcal{P}$, we denote by $D_i$ 
the irreducible component of $B(\mathcal{P})$ corresponding to $\mathcal{F}_i$, 
and refer to $D_i$ as the \emph{boundary component associated with $\mathcal{F}_i$}.}
\end{definition}

When the choice of $\mathcal{P}$ is clear, we will simply denote the tropical mutation boundary as $B$. From now on, we will also assume all polyptych lattices are degenerable. Furthermore, every facet is defined by an equation of the form $p_i-\alpha_i=0$ for a tropical point $p_i\in \operatorname{Sp}(\mathcal{M})$ and $\alpha_i\in \mathbb{Z}_{<0}.$ 

Let $X_{A_{\mathcal{M}}}(\mathcal{P})$ be a tropical mutation variety whose affine tropical mutation variety $U_{A_{\mathcal{M}}}$ is a normal variety, and $\mathcal{P}$ is a normal polytope (cf.~\cite[Definition~7.13]{EHM24})\footnote{Throughout this article we assume polytopes are normal. Note, every two-dimensional integral polytope is normal~\cite[Corollary~2.2.13]{CLS11}, and Proposition~\ref{classification} shows that $A_{\mathcal{M}_s}$ is normal under any detropicalization.}. Then $X_{A_{\mathcal{M}}}(\mathcal{P})$ is embedded into projective space by the adapted basis. 

\begin{proposition}\label{prop:normal}
{\rm (cf.~\cite[Lemma 7.14 and Proposition 7.16]{EHM24})} 
Let $\mathcal{M}$ be a polyptych lattice and $\mathcal{N}$ its strict dual $\operatorname{w}\colon \mathcal{N}\rightarrow \operatorname{Sp}(\mathcal{M})$. Let $\mathcal{P} \subset \mathcal{M}$ be a normal polytope defined by the tropical points $\operatorname{w}(n_1) = \alpha_1, \dots, \operatorname{w}(n_k) = \alpha_k$. Let $D_i$ be the divisor associated with the tropical point $\operatorname{w}(n_i)$. If $A_{\mathcal M}$ is normal, then the following statements hold:
\begin{enumerate}
\item $X_{A_\mathcal{M}}(\mathcal P)$ is normal;
\item the valuation $\operatorname{ord}_{D_i}\colon A_{\mathcal M}\!\setminus\!\{0\}\to\mathbb Z$ coincides with the composition $\mathfrak{v}\circ n_i$; and
\item for any integral polytope $\mathcal P$, the algebra $A^{\mathcal P}_{\mathcal M}$ is generated in degree~\(\!1\). 
\end{enumerate}
\end{proposition}

Tropical mutation varieties admit several toric degenerations over $\mathbb{A}^1$. In Theorem~\ref{toricdegene} below, we combine these degenerations into a single family defined over projective space.
\begin{theorem}\label{toricdegene}
Let
$
\mathcal M = \left( (M_\alpha)_{\alpha\in I}, (\mu_{\alpha,\beta})_{\alpha,\beta\in I} \right)
$
be a degenerable polyptych lattice with $|I|=m$. Let $\mathcal P\subset \mathcal M$ be a normal polytope, and let $P_\alpha$ be the image of $\mathcal P$ in the chart $M_\alpha$. Assume that the multi-Rees algebra associated with the homogeneous chart valuations $(\widetilde{\mathfrak v}_\alpha)_{\alpha\in I}$ is finitely generated. Then there exists a flat projective family
$
\mathfrak X_{A_{\mathcal M}}(\mathcal P) \longrightarrow \mathbb P^{m-1}
$
such that:
\begin{enumerate}
    \item its restriction over the dense torus $\mathbb G_m^{m-1}\subset \mathbb P^{m-1}$ is the trivial family with fiber $X_{A_{\mathcal M}}(\mathcal P)$;
    \item for each $\alpha\in I$, the fiber over a general point of the coordinate hyperplane
    $
    H_\alpha = \{ \tau_\alpha=0 \} \subset \mathbb P^{m-1}
    $
    is isomorphic to the toric variety $T(P_\alpha)$.
\end{enumerate}
\end{theorem}

\begin{proof}
Let
$
A_{\mathcal M}^{\mathcal P} = \bigoplus_{d\geq 0} \Gamma(U_{A_{\mathcal M}}, d\mathcal P)t^d
$
be the polytope algebra. For each chart $\alpha\in I$, let $\widetilde{\mathfrak v}_\alpha$ be the homogeneous chart valuation of \cite[Lemma~7.8]{EHM24}. We normalize these valuations so that their values on each graded piece of $A_{\mathcal M}^{\mathcal P}$ are nonnegative.

For $\mathbf k=(k_\alpha)_{\alpha\in I}\in\mathbb Z_{\geq 0}^I$, set
$F_{\leq \mathbf k}A_{\mathcal M}^{\mathcal P} := \left\{ s\in A_{\mathcal M}^{\mathcal P} \mid\,\widetilde{\mathfrak v}_\alpha(s)\leq k_\alpha \text{ for every } \alpha\in I \right\}.$
These subspaces define a multiplicative multi-filtration. Indeed, if $s\in F_{\leq \mathbf k}$ and $s'\in F_{\leq \mathbf l}$, then
$
\widetilde{\mathfrak v}_\alpha(ss') \leq \widetilde{\mathfrak v}_\alpha(s) + \widetilde{\mathfrak v}_\alpha(s') \leq k_\alpha+l_\alpha
$
for every $\alpha\in I$. Define the multi-Rees algebra
$$
\mathcal R_{A_{\mathcal M}}^{\mathcal P} := \bigoplus_{\mathbf k\in\mathbb Z_{\geq 0}^I} F_{\leq \mathbf k}A_{\mathcal M}^{\mathcal P}\cdot \tau^{\mathbf k} \subset A_{\mathcal M}^{\mathcal P}[\tau_\alpha\mid \alpha\in I], \qquad \tau^{\mathbf k}:=\prod_{\alpha\in I}\tau_\alpha^{k_\alpha}.
$$
By assumption, $\mathcal R_{A_{\mathcal M}}^{\mathcal P}$ is finitely generated over $\mathbb K[\tau_\alpha\mid \alpha\in I]$. We regard $\mathcal R_{A_{\mathcal M}}^{\mathcal P}$ as graded by the original polytope degree. Thus
$
\mathcal R_{A_{\mathcal M}}^{\mathcal P} = \bigoplus_{d\geq 0} \left(\mathcal R_{A_{\mathcal M}}^{\mathcal P}\right)_d,
$
where
$
\left(\mathcal R_{A_{\mathcal M}}^{\mathcal P}\right)_d = \bigoplus_{\mathbf k\in\mathbb Z_{\geq 0}^I} F_{\leq \mathbf k} \left(A_{\mathcal M}^{\mathcal P}\right)_d \tau^{\mathbf k}.
$
Each piece $\left(\mathcal R_{A_{\mathcal M}}^{\mathcal P}\right)_d$ is still a graded $\mathbb K[\tau_\alpha\mid \alpha\in I]$-module with respect to the standard $\tau$-grading.

By the adapted-basis property, the vector space $\left(A_{\mathcal M}^{\mathcal P}\right)_d$ has basis
$\theta_q t^d, \,q\in d\mathcal P\cap \mathcal M,$
where $q$ ranges over the tropical integral points of $d\mathcal P$. The homogeneous chart valuations are diagonal with respect to this basis. Write
$
w_\alpha(q,d) := \widetilde{\mathfrak v}_\alpha(\theta_q t^d).
$
Then
$
\theta_q t^d \in F_{\leq \mathbf k} \left(A_{\mathcal M}^{\mathcal P}\right)_d \quad \Longleftrightarrow \quad w_\alpha(q,d)\leq k_\alpha \text{ for every } \alpha\in I.
$
Hence we have a direct sum decomposition
$$
\left(\mathcal R_{A_{\mathcal M}}^{\mathcal P}\right)_d = \bigoplus_{q\in d\mathcal P\cap \mathcal M} \left( \prod_{\alpha\in I} \tau_\alpha^{w_\alpha(q,d)} \right) \mathbb K[\tau_\alpha\mid \alpha\in I]\cdot \theta_q t^d.
$$
In particular, each polytope-degree piece $\left(\mathcal R_{A_{\mathcal M}}^{\mathcal P}\right)_d$ is a free $\mathbb K[\tau_\alpha\mid \alpha\in I]$-module, up to the indicated $\tau$-degree shifts.

Now set
$
S := \mathbb P^{m-1} = \operatorname{Proj}\Bbbk[\tau_\alpha\mid \alpha\in I].
$
For each $d$, let
$
\mathscr R_d := \widetilde{ \left(\mathcal R_{A_{\mathcal M}}^{\mathcal P}\right)_d }
$
be the quasi-coherent sheaf on $S$ associated to the $\tau$-graded $\Bbbk[\tau_\alpha\mid \alpha\in I]$-module $\left(\mathcal R_{A_{\mathcal M}}^{\mathcal P}\right)_d$. The multiplication on the Rees algebra gives
$
\mathscr R := \bigoplus_{d\geq 0}\mathscr R_d
$
the structure of a graded $\mathcal O_S$-algebra, where the grading is the polytope degree. Since the Rees algebra is finitely generated, $\mathscr R$ is a finitely generated graded $\mathcal O_S$-algebra. Define
$$
\mathfrak X_{A_{\mathcal M}}(\mathcal P) := \operatorname{Proj}_S \mathscr R.
$$
Projectivity over $S$ follows from the relative Proj construction. Moreover, by the decomposition above, each $\mathscr R_d$ is a finite direct sum of line bundles on $S$. Hence each $\mathscr R_d$ is locally free, and its rank is independent of the point of $S$. After replacing $\mathscr R$ by a sufficiently divisible Veronese subalgebra, which does not change the relative Proj, we may assume that $\mathscr R$ is generated in degree one.

The Hilbert function of the fibers with respect to the relative $\mathcal O(1)$ is then constant, since it is given by the ranks of the locally free sheaves $\mathscr R_d$. Therefore, by the standard flatness criterion for projective morphisms with constant Hilbert polynomial,
$
\mathfrak X_{A_{\mathcal M}}(\mathcal P) \longrightarrow S
$
is flat. 

It remains to identify the fibers. Over the dense torus $\mathbb G_m^{m-1}\subset S$, all parameters $\tau_\alpha$ are invertible. Thus all Rees homogenizing factors become units, and we obtain
$$
\mathcal R_{A_{\mathcal M}}^{\mathcal P} [\tau_\alpha^{-1}\mid \alpha\in I] \cong A_{\mathcal M}^{\mathcal P} [\tau_\alpha^{\pm 1}\mid \alpha\in I].
$$
Consequently, the restriction of the family over the dense torus is the trivial family with fiber $X_{A_{\mathcal M}}(\mathcal P)$.

Now fix $\alpha\in I$. Consider the open subset
$
H_\alpha^\circ := H_\alpha\cap \bigcap_{\beta\neq \alpha} \{ \tau_\beta\neq 0 \} \subset H_\alpha.
$
Over $H_\alpha^\circ$, we invert all $\tau_\beta$ with $\beta\neq \alpha$ and then set $\tau_\alpha=0$. This gives an isomorphism of graded algebras
$$
\mathcal R_{A_{\mathcal M}}^{\mathcal P} [\tau_\beta^{-1}\mid \beta\neq \alpha]/(\tau_\alpha) \cong \operatorname{gr}_{\widetilde{\mathfrak v}_\alpha} \left(A_{\mathcal M}^{\mathcal P}\right) [\tau_\beta^{\pm 1}\mid \beta\neq \alpha].
$$
By the toric degeneration theorem of Escobar--Harada--Manon, equivalently by \cite[Proposition~5.1]{MR3063911} together with \cite[Lemma~7.9]{EHM24}, the associated graded algebra
$
\operatorname{gr}_{\widetilde{\mathfrak v}_\alpha} \left(A_{\mathcal M}^{\mathcal P}\right)
$
is, as a graded algebra, the semigroup algebra of the chart polytope $P_\alpha$. Therefore
$
\operatorname{Proj} \operatorname{gr}_{\widetilde{\mathfrak v}_\alpha} \left(A_{\mathcal M}^{\mathcal P}\right) \cong T(P_\alpha).
$
Thus every fiber over $H_\alpha^\circ$, and in particular the fiber over a general point of $H_\alpha$, is isomorphic to $T(P_\alpha)$.
\end{proof}

Based on the theorem above, we define the \emph{global tropical mutation variety}. 

\begin{definition}\label{def:GTMV}
\emph{Let $\mathcal{M}$ be a polyptych lattice with $m$ charts, and let 
$\mathcal{P} \subset \mathcal{M}$ be a polytope.  
Define the Rees algebra
\[
\mathcal{R}^{\mathcal{P}}_{A_\mathcal{M}} 
:= \bigoplus_{\mathbf{k}\in\mathbb{Z}_{\geq 0}^m} 
F_{\leq \mathbf{k}} A^{\mathcal{P}}_{\mathcal{M}}\tau^{\mathbf{k}},
\]
associated with the multi–index filtration $F_{\leq \mathbf{k}}$ constructed above.  
The \emph{global tropical mutation variety associated with $\mathcal{P}$} is the relative Proj, 
 $\mathcal{X}_{A_\mathcal{M}}(\mathcal{P}) := \operatorname{Proj}_{\mathbb{P}^{m-1}}\!\big(\mathcal{R}^{\mathcal{P}}_{A_\mathcal{M}}\big).$}
\end{definition}

In this article, we mainly study tropical mutation surfaces arising from the polyptych lattice with a single shear, introduced by Cook–Escobar–Harada–Manon~\cite{CEHM24}. We refer to the tropical mutation surface associated with this polyptych lattice as a \emph{shearing tropical mutation surface.}

\begin{definition}\label{EHM}
{\rm (cf.~\cite[Section~3]{CEHM24})}
\emph{For $s\in\mathbb Z_{>0}$\footnote{The polyptych lattice $\mathcal{M}_s$ is defined for any integer $s$. However, $\mathcal{M}_s$ and $\mathcal{M}_{-s}$ are strongly isomorphic. Hence for simplicity of notation, we restrict to the case $s>0$. $\mathcal{M}_0$ is strongly isomorphic to the trivial lattice.}
the \emph{shearing polyptych lattice} $\mathcal{M}_s$ is defined by:
\[
\mathcal M_s := (M_1,M_2,\mu_{1,2}),
\qquad
\mu_{1,2}(x,y)=
\begin{cases}
(-x,\,y),      & y\ge 0,\\
(sy-x,\,y),    & y\le 0.
\end{cases}
\]
We write an element of $\mathcal{M}_s$ as $((x,y), (x',y))\in M_1\times M_2$ with $x+x'=\min\{0,sy\}$.
Since $\mathcal{M}_s$ is self-dual, every tropical point $p\in \operatorname{Sp}(\mathcal{M}_s)$ is represented by itself. Therefore, we define a tropical point 
$\operatorname{w}((a_i,c_i),(b_i,c_i))\colon\mathcal{M}_s\rightarrow \mathbb{Z}$. For $((x,y), (x',y))\in \mathcal{M}_s$:
\[
\operatorname{w}((a_i,c_i),(b_i,c_i))((x,y), (x',y))=
\begin{cases}
c_ix - b_iy, & y\le 0,\\
c_ix + a_iy, & y\ge 0,
\end{cases} 
\qquad \text{where}\qquad a_i+b_i=\min\{s c_i,0\}.\]}
\end{definition}

\section{Classification of affine tropical mutation varieties associated with \texorpdfstring{$\mathcal M_s$}{Ms}}

We classify the possible detropicalizations associated with the shearing
polyptych lattice \(\mathcal M_s\). 

\begin{proposition}\label{classification}
The detropicalizations of \(\mathcal M_s\) are precisely the hypersurface
algebras $A_f =
        \mathbb K[x_1,x_2,y^{\pm1}]/(x_1x_2-f(y)),$
where \(f(y)\in\mathbb K[y]\) has degree \(s\) and \(f(0)\neq 0\). More
precisely, every detropicalization of \(\mathcal M_s\) is isomorphic to
\(A_f\) for such a polynomial \(f\), and conversely every such polynomial
defines a detropicalization of \(\mathcal M_s\).

Let $f(y)=\prod_{i=1}^{\gamma}(y-\alpha_i)^{\beta_i},\, \alpha_i\in\mathbb K^*,$
then \(U_f\) has a Du Val singularity of type \(A_{\beta_i-1}\) at
\((0,0,\alpha_i)\). In particular, \(U_f\) is smooth at this point when
\(\beta_i=1\), and \(A_f\) is normal.
\end{proposition}

\begin{proof}
The equivalence between detropicalizations of \(\mathcal M_s\) and the
hypersurface algebras \(A_f\) follows from \cite[Theorem~6.20]{CEHM24}. The
construction there associates to a mutation factor \(f(y)\) the algebra
 $\mathbb K[x_1,x_2,y^{\pm1}]/(x_1x_2-f(y)),$
and the same argument shows conversely that every polynomial
\(f(y)\in\mathbb K[y]\) of degree \(s\) with \(f(0)\neq 0\) gives a
detropicalization of the shearing polyptych lattice \(\mathcal M_s\).

It remains to record the singularities. Since \(\mathbb K[x_1,x_2,y^{\pm1}]\) is the coordinate ring of \(\mathbb A^2\times\mathbb G_m\), the algebra \(A_f\) is a two-dimensional hypersurface ring, hence Cohen--Macaulay. The Jacobian criterion gives  $\operatorname{Sing}(U_f) = \{(0,0,\alpha)\mid f(\alpha)=f'(\alpha)=0\}. $ Near a root \(\alpha_i\) of multiplicity \(\beta_i\), writing \(t=y-\alpha_i\), the completed local equation is analytically equivalent to $x_1x_2=t^{\beta_i}.$ Thus the singularity is of type \(A_{\beta_i-1}\), and it is smooth when \(\beta_i=1\). Since the singular locus is finite, \(A_f\) satisfies \(R_1\); being a hypersurface, it satisfies \(S_2\). Hence \(A_f\) is normal by Serre's criterion.
\end{proof}

Next, we give a criterion for when two polynomials \(f\) and \(g\) produce
isomorphic affine tropical mutation varieties \(U_f \cong U_g\).

\begin{proposition}\label{prop:iso_and_aut}
Let \(f,g\in \mathbb K[y]\) have degree \(s\) and nonzero constant terms.
Then \(U_f\) and \(U_g\) are isomorphic if and only if there exist
\(\lambda,c\in \mathbb K^*\) such that
\[
        g(y)=\lambda f(cy)
        \qquad\text{or}\qquad
        g(y)=\lambda y^s f(c/y).
\]
Let \(f\) be normalized, say $f(y)=y^s+b_{s-1}y^{s-1}+\cdots+b_1y+1.$
Then
\[
        \operatorname{Aut}(U_f)
        \cong
        (\mathbb G_m\times \mathbb Z)
        \rtimes
        \bigl(\mu_2\times \operatorname{Aut}(f)\bigr),
\]
where \(\mu_2\) interchanges \(x_1\) and \(x_2\), and
\(\operatorname{Aut}(f)\leq D_{2s}\) is the subgroup consisting of pairs
\((\epsilon,c)\in \{0,1\}\times \mu_s\) satisfying $f(y)=y^{s\epsilon}f(cy^{(-1)^\epsilon}).$
\end{proposition}

\begin{proof}
Let $A_f=\mathbb K[x_1,x_2,y^{\pm1}]/(x_1x_2-f(y)).$
First we observe that $A_f^*=\mathbb K^* y^{\mathbb Z}.$
Indeed, using the natural \(\mathbb Z\)-grading
\(\deg x_1=1\), \(\deg x_2=-1\), and \(\deg y=0\), the highest and lowest
degree terms of a unit force the unit to be homogeneous of degree zero.
Thus every unit lies in \(\mathbb K[y^{\pm1}]^*\).

Hence any isomorphism \(\Phi:A_f\to A_g\) sends  $y\longmapsto cy^{\pm1}$
for some \(c\in\mathbb K^*\). After composing with the corresponding
automorphism of the Laurent coordinate, it remains to understand
isomorphisms over \(\mathbb K[y^{\pm1}]\).

Let \(K_0=\mathbb K(y)\). After tensoring with \(K_0\), both algebras become
Laurent polynomial algebras:
\[
        A_f\otimes_{\mathbb K[y^{\pm1}]}K_0
        \simeq K_0[x_1^{\pm1}],
        \qquad
        A_g\otimes_{\mathbb K[y^{\pm1}]}K_0
        \simeq K_0[x_1^{\pm1}].
\]
Therefore the induced automorphism of the generic fiber sends
\[
        x_1\longmapsto a(y)x_1
        \qquad\text{or}\qquad
        x_1\longmapsto a(y)x_2
\]
for some \(a(y)\in K_0^*\). Since both \(\Phi\) and \(\Phi^{-1}\) are regular
on the whole surface, this multiplier must be a Laurent unit. Thus
\(a(y)\in \mathbb K^*y^{\mathbb Z}\). Consequently \(\Phi\) sends
\(\{x_1,x_2\}\) to Laurent-monomial multiples of \(\{x_1,x_2\}\), possibly
interchanging them.

Applying \(\Phi\) to the defining relation now gives the desired condition.
If \(y\mapsto cy\), then $g(y)=\lambda f(cy)$
for some \(\lambda\in\mathbb K^*\). If \(y\mapsto c/y\), then clearing the
negative powers of \(y\) gives $g(y)=\lambda y^s f(c/y).$
Conversely, either formula gives an explicit isomorphism by the corresponding
change of the Laurent coordinate \(y\), together with a Laurent-monomial
rescaling of \(x_1\) and \(x_2\). This proves the classification of
isomorphisms.

Setting \(f=g\), the same description gives the automorphism group. The
factor \(\mathbb G_m\times\mathbb Z\) comes from $x_1\mapsto \lambda y^n x_1,\quad
        x_2\mapsto \lambda^{-1}y^{-n}x_2.$
The factor \(\mu_2\) is generated by the involution interchanging \(x_1\)
and \(x_2\). Finally, \(\operatorname{Aut}(f)\leq D_{2s}\) records the
allowed automorphisms of the Laurent coordinate \(y\mapsto cy\) and
\(y\mapsto c/y\) preserving the normalized polynomial \(f\). This gives the
claimed semidirect product description.
\end{proof}

\begin{theorem}\label{theorem:moduli}
The coarse moduli space of detropicalizations is isomorphic to \(\mathbb A^{s-1}/D_{2s}\).
\end{theorem}

\begin{proof}
The normalized mutation factors are parametrized by
\(\mathbb A^{s-1}/D_{2s}\) up to isomorphism.

By Proposition~\ref{prop:iso_and_aut}, locally in the \'etale topology on $S$, we can scale the coordinates of the family such that the leading and constant coefficients are equal to $1$. This local reduction isolates the $s-1$ intermediate coefficients as regular functions, inducing a local morphism $U \to \mathbb{A}^{s-1}$.

The residual equivalences between any two such local reductions are exactly:
\begin{enumerate}
    \item Scaling $y$ by $\zeta \in \mu_s$, which acts as $\zeta \cdot (b_1,\dots,b_{s-1}) = (\zeta b_1,\dots,\zeta^{s-1}b_{s-1})$.
    \item The transformation $f(y) \mapsto y^s f(1/y)$, which reverses the coefficients $(b_{s-1},\dots,b_1)$.
\end{enumerate}

These transformations precisely generate the standard action of the dihedral group $D_{2s}$ on $\mathbb{A}^{s-1}$. Because the local morphisms to $\mathbb{A}^{s-1}$ differ only by this $D_{2s}$ action, they uniquely descend and glue to form a well-defined global morphism $S \to \mathbb{A}^{s-1}/D_{2s}$. This induced map satisfies the universal property, establishing $\mathbb{A}^{s-1}/D_{2s}$ as the coarse moduli space.
\end{proof}
Consequently, the natural $\mathbb G_m$-action extends equivariantly to every projective tropical mutation surface.
\begin{corollary}
\label{torusmu_delpezzo}
\em{Let $X_{f}(\mathcal{P})$ be the projective tropical mutation surface associated to a polynomial $f$ and a polytope $\mathcal{P}\subset \mathcal{M}_s$. The natural action of $\mathbb{G}_m$ on $U_f$ uniquely extends to $X_f(\mathcal{P})$. Consequently, if $X_f(\mathcal{P})$ is a smooth del Pezzo surface, it must be toric.}
\end{corollary}

\begin{proof}
Let $G = \mathbb{G}_m \times \operatorname{Aut}(f)$. The $G$-action on the coordinate ring $A_f$, given by 
$$x_1 \mapsto \lambda x_1, \qquad x_2 \mapsto \lambda^{-1} x_2, \qquad y \mapsto cy^{\epsilon},$$
preserves the defining relation $x_1x_2 = f(y)$. 

Because this action scales or symmetrically interchanges Laurent monomials, it preserves the boundary valuations defining $\mathcal{P}$. Thus, the finite-dimensional linear system of sections $V_{\mathcal{P}} \subset A_f$ is a $G$-invariant subspace. The resulting linear representation of $G$ on $V_{\mathcal{P}}$ naturally induces an algebraic action on $\mathbb{P}(V_{\mathcal{P}})$, which restricts to its Zariski closure, $X_f(\mathcal{P})$.

By the classification of smooth del Pezzo surfaces, non-toric ones have strictly finite automorphism groups, forcing $X_f(\mathcal{P})$ to be toric.
\end{proof}

\section{Singularities of tropical mutation varieties}\label{sec:sing}
In this section, we study the singularities of tropical mutation varieties.

\begin{definition}
A log pair $(X, B)$ is called a \emph{log Calabi–Yau pair} if $K_X + B \sim 0$ and $(X, B)$ is log canonical.  
\end{definition}
We now study the singularities of projective tropical mutation surfaces.
\begin{proposition}\label{sm}
Let $\mathcal{P}$ be a normal polytope in $\mathcal{M}$, with tropical mutation surfaces
$X_{A_\mathcal{M}}(\mathcal{P})$ and boundary divisor $B=\sum D_i$. Let
$\mathcal{X}_{\alpha}\to\mathbb{A}^1$ be the degeneration to the toric variety
$T(P_{\alpha})$. Then:
\begin{enumerate}
\item Each boundary component $D_i$ specializes to a union of toric boundary components of $T(P_{\alpha})$.
\item $X_{A_\mathcal{M}}(\mathcal{P})$ has rational singularities. Moreover, if the corresponding local cones of the chart polytope $P_{\alpha}$ are smooth
\emph{(resp. Gorenstein)}, then $X_{A_\mathcal{M}}(\mathcal{P})$ is smooth
\emph{(resp. Gorenstein)} along the corresponding boundary strata.
\item Boundary singularities of $X_{A_\mathcal{M}}(\mathcal{P})$ can occur only at the strata corresponding to PL vertices of $\mathcal{P}$. Equivalently, the smooth loci of the boundary components away from the PL vertices are smooth points of $X_{A_\mathcal{M}}(\mathcal{P})$.
\item $(X_{A_{\mathcal{M}}}(\mathcal{P}),B)$ is an index $1$ log Calabi--Yau pair.
\item When $\mathcal{M}$ has rank two, the boundary singularities are at most cyclic quotient singularities.
\end{enumerate}
\end{proposition}

\begin{proof}
For (1), the adapted-basis embedding $\Phi_{\mathcal{P}}$ degenerates to the corresponding toric embedding. By \cite[Theorem 1]{KavehManon}, the adapted basis elements $\theta_{m_j}$ specialize to the toric characters $\chi_{m_{j,\alpha}}$. Therefore the vanishing loci defining the tropical mutation boundary specialize to unions of toric boundary strata in the central toric fiber $T(P_{\alpha})$. In particular, each $D_i$ specializes to a union of toric boundary components.

We next prove (2) and (3). The boundary of $X_{A_\mathcal{M}}(\mathcal{P})$ is covered by the affine charts determined by the facets and vertices of $\mathcal{P}$. On such a chart, the adapted-basis degeneration identifies the associated graded local algebra with the affine semigroup algebra of the corresponding toric cone in the chart polytope $P_{\alpha}$. Thus the local boundary singularities of $X_{A_\mathcal{M}}(\mathcal{P})$ are controlled by the same cones that control the boundary singularities of the toric model $T(P_{\alpha})$. Normal toric varieties have rational singularities, and smoothness and the Gorenstein property are determined by the corresponding cones; see \cite[Section 8.2]{CLS11}. Hence $X_{A_\mathcal{M}}(\mathcal{P})$ has rational singularities along the boundary, and it is smooth \emph{(resp. Gorenstein)} along the boundary strata whose corresponding local cones are smooth \emph{(resp. Gorenstein)}. Moreover, away from the vertices of $\mathcal{P}$, the boundary is locally modeled on the product of a smooth boundary divisor with a transverse parameter. Therefore no boundary singularity occurs at the general point of a boundary component. Possible boundary singularities occur only at the zero-dimensional boundary strata, equivalently at the strata corresponding to PL vertices of $\mathcal{P}$. This proves (3)~\cite[Chapter~9]{shi}.

For (4), the open complement $U_{A_\mathcal{M}}=X_{A_\mathcal{M}}(\mathcal{P})\setminus B$ carries the canonical logarithmic volume form coming from the detropicalized torus chart. Its divisor has simple logarithmic poles along the reduced boundary $B$ and no other zeros or poles. Consequently $K_{X_{A_\mathcal{M}}(\mathcal{P})}+B\sim 0.$
The local description above shows that the pair is toroidal along the boundary, up to the rational toric singularities described by the local cones. Hence $(X_{A_\mathcal{M}}(\mathcal{P}),B)$ is log canonical. Therefore $(X_{A_\mathcal{M}}(\mathcal{P}),B)$ is a log Calabi--Yau pair. This is compatible with the toric degeneration, whose central fiber is the toric log Calabi--Yau pair $(T(P_{\alpha}),B_{T(P_{\alpha})})$; cf.~\cite[Lemma 8.42]{GHKK}.

Finally, assume that $\mathcal{M}$ has rank two. Then the boundary singularities are surface singularities locally modeled on two-dimensional toric cones. Every normal affine toric surface singularity is a cyclic quotient singularity. Equivalently, by the classification of two-dimensional log canonical pairs with reduced boundary, such boundary singularities are cyclic quotient singularities; see \cite[Section 3.40]{Kol13}. This proves (5).
\end{proof}
\begin{proposition}\label{MDS} 
$\mathbb Q$-factorial tropical mutation surfaces are of Fano type. In particular, they are Mori dream spaces.
\end{proposition}
\begin{proof}
Let $(X,B)$ be an index-one projective tropical mutation surface pair. By Proposition~\ref{sm}, $(X,B)$ is $\mathbb Q$-factorial log canonical and 
$K_X+B\sim 0.$ 
Furthermore, by Theorem~\ref{Ample}(4), the boundary $B$ supports an effective ample $\mathbb Q$-divisor 
$A=\sum_i a_iD_i,$ $ a_i>0.$

On the interior $U=X\setminus B$, the pair simplifies to $(U,0)$. The index-one assumption on $(X,B)$ means that the Weil divisor $K_X+B$ is globally Cartier. Restricting to the interior where $B$ is empty, it follows that the canonical divisor $K_U = (K_X+B)|_U$ remains Cartier. Since log canonical surfaces are Cohen-Macaulay, the interior having a Cartier canonical divisor guarantees it has Gorenstein singularities. Because $X$ is $\mathbb Q$-factorial and $(X,B)$ is log canonical, these interior singularities are therefore Gorenstein klt. Their discrepancies are integers strictly greater than $-1$, meaning they are nonnegative; thus, they are canonical and consequently Du Val \cite[Theorem~4.20]{KM98}. As a result, any non-klt centers of $(X,B)$ must lie on the boundary.

A single toroidal blow-up at each non-dlt nodal stratum provides a dlt modification $\mu:(Y,B_Y)\longrightarrow (X,B)$ satisfying $K_Y+B_Y=\mu^*(K_X+B)$. Because $A$ has positive coefficients along $B$, its pullback $\mu^*A$ has positive coefficients along $B_Y$. Therefore, for $0<\epsilon\ll 1$, the pair $(Y,B_Y-\epsilon\mu^*A)$ is klt. By setting $\Delta:=B-\epsilon A$, we obtain $K_Y+B_Y-\epsilon\mu^*A = \mu^*(K_X+\Delta)$, which proves that $(X,\Delta)$ is klt.

Finally, we have
$-(K_X+\Delta) = -(K_X+B)+\epsilon A \sim_{\mathbb Q} \epsilon A,$
which is ample. Hence, $(X,\Delta)$ is a klt log Fano pair. This implies $X$ is of Fano type and, consequently, a Mori dream space \cite[Corollary~1.3.1]{BCHM10}.
\end{proof}

\section{The complexity of tropical mutation surface pairs}
\label{sec:cmp}

 In this section, we compute the complexity of the tropical mutation pair $(X_f(\mathcal{P}), B)$ and prove that it is a cluster type pair. The key geometric observation is that the roots of \(f\) give distinguished
interior curves on \(U_f\). After compactification, these curves meet either
the source or the sink of the \(G_m\)-action. Resolving the interior Du Val
singularities and contracting these curves recovers a toric model. This
birational factorization is the mechanism behind both the cluster-type
property and the complexity formula.
\begin{definition}{\rm (cf.~\cite[Definition 2.14]{EFM24})}\label{def:complexity}
\emph{Let $X$ be a projective $\mathbb Q$-factorial variety with Picard rank
$\rho(X)$, and let $B=\sum B_i$ be a Weil divisor such that
$(X,B)$ is log canonical and $-(K_X+B)$ is nef. The
\emph{complexity} of the pair $(X,B)$ is
$ c(X,B):=\dim X+\rho(X)-|B|, $
where $|B|$ denotes the sum of the coefficients of $B$.}
\end{definition}

We describe two elementary configurations of boundary points which will be used below.

\begin{definition}
\label{def:nodal-collinear}
\emph{Let $(X,B)$ be a projective log Calabi--Yau surface pair.
\begin{enumerate}
  \item A point $p\in X$ is called \emph{nodal} if it lies in the intersection
  of two irreducible components of $B$.
  \item A collection of non-nodal points $p_1,\dots,p_s$ is called
  \emph{collinear} if all $p_i$ lie on the same irreducible component of
  $B$. We allow repetitions when keeping track of infinitely near 
  centers.
\end{enumerate}}
\end{definition}

The following lemma explains why collinear blow-ups increase the complexity,
whereas nodal toric modifications preserve it.
\begin{lemma}
\label{lem:blow-up-boundary}
Let $(T,B_T)$ be a projective toric surface with a reduced toric boundary. Let $\pi\colon \widetilde T\longrightarrow T$ be the blow-up of $s$ collinear boundary points, and let $\widetilde B$ be the strict transform of $B_T$. Then:
\begin{enumerate}
    \item $(\widetilde T,\widetilde B)$ is a log Calabi--Yau pair with $c(\widetilde T,\widetilde B)=s$;
    \item the one-dimensional subtorus fixing the boundary component containing the points $p_i$ lifts to $\widetilde T$;
    \item toric weighted blow-ups at nodal boundary points preserve the complexity, provided the exceptional divisor is included in the boundary.
\end{enumerate}
\end{lemma}

\begin{proof}
Because $(T,B_T)$ is toric, $K_T+B_T\sim 0$. The ordinary blow-ups at the smooth locus of $B_T$ are crepant, so $K_{\widetilde T}+\widetilde B=\pi^*(K_T+B_T) \sim 0$, making $(\widetilde T,\widetilde B)$ log Calabi--Yau. Since blowing up $s$ points increases the Picard rank by $s$ while leaving the number of boundary components unchanged, the complexity is $c(\widetilde T,\widetilde B) = 2+(\rho(T)+s)-|B_T| = s$. 

Let $D\subset B_T$ contain the points $p_i$. The dense orbit of $D$ is fixed by a one-dimensional subtorus; since each $p_i$ lies on this orbit, the action canonically lifts to $\widetilde T$. Finally, a toric weighted blow-up at a node extracts exactly one new boundary divisor, increasing both $\rho$ and $|B|$ by $1$. The complexity $2+\rho-|B|$ is thus invariant, and the log Calabi--Yau property is preserved.
\end{proof}

We now describe the distinguished interior curves on the affine tropical mutation surface. This allows us to compute the complexity from it.

\begin{definition}\label{def:source-sink}
\emph{For each root $\alpha_r$ of $f$, the affine surface $U_f=\{x_1x_2=f(y)\}$ contains two \emph{distinguished interior curves} defined as $C_{2r-1}=\{x_1=0,\ y=\alpha_r\}$ and $C_{2r}=\{x_2=0,\ y=\alpha_r\}$.}

\emph{For a polytope $\mathcal P\subset \mathcal M_s$, define $\mathcal F_{\max}$ and $\mathcal F_{\min}$ as the maximum and minimum faces the height function $h$ on $\mathcal P$. The corresponding boundary strata of $X_f(\mathcal P)$ are called the \emph{sink} and \emph{source}, respectively.}
\end{definition}

Since $h$ is linear on the chosen chart, these loci are faces of $\mathcal P$. Let $h_{\max}:=\max_{\mathcal P}h$ and define $r_{\max} := \min\{h_{\max}-h(m)>0 \mid m\in\mathcal P\cap M\}$. We define $h_{\min}$ and $r_{\min}$ analogously. 

\begin{proposition}
\label{prop:source-sink-transversality}
Let $X_f(\mathcal P)$ be the tropical mutation surface with boundary $B$. The closure $\overline{C}_{2r}$ meets the sink, and $\overline{C}_{2r-1}$ meets the source. Furthermore:
\begin{enumerate}
    \item If the sink (resp.\ source) is a boundary divisor $D_{\max}$ (resp.\ $D_{\min}$), and assuming the relevant adapted coordinate restricts nontrivially, $\overline C_{2r}$ meets $D_{\max}$ with contact order $r_{\max}$ (resp.\ $\overline C_{2r-1}$ meets $D_{\min}$ with contact order $r_{\min}$). This intersection is transverse when the minimal difference is $1$.
    \item If the sink or source is a nodal stratum, a toric weighted blow-up along the primitive contact vector extracts an exceptional divisor $E_{v_0}$. The strict transforms of the distinguished curves intersect $E_{v_0}$ away from the old strata, and transversally if and only if their initial contact vector was primitive. 
    \item A single weighted blow-up separates all distinguished curves approaching a node with the same primitive contact direction, meeting $E_{v_0}$ at distinct non-nodal points.
\end{enumerate}
\end{proposition}

\begin{proof}
We first prove (1). Let $i=2r$ be even. By definition, $C_{2r}$ is the locus where $x_2=0$. 
Using the adapted basis $\{\theta_0, \dots, \theta_m\} \subset \Gamma(U_f, \mathcal{P})$, 
we embed $C_{2r}$ into the projective space $\mathbb{P}^m$, 
where the coordinate $\theta_0$ corresponds to the adapted basis element for the origin of $\mathcal{P}$. In the affine chart $\theta_0=1$, we have
\begin{equation}
C_{2r}(t) = \bigl(1, \theta_1(t, 0, \alpha_r), \dots, \theta_m(t, 0, \alpha_r)\bigr),\quad t\in\mathbb{A}^1.
\end{equation}

Let $k$ be the maximal $y$-coordinate of a lattice point of $\mathcal{P}$. Passing to homogeneous coordinates $[\tau_0:\tau_1]$ on $\mathbb{P}^1$ by setting $t = \tau_0/\tau_1$ and multiplying by $\tau_1^k$, we obtain
\begin{equation}\label{eq:22}
C_{2r}([\tau_0:\tau_1]) = \left[\tau_1^k : \tau_1^k\theta_1\left(\frac{\tau_0}{\tau_1}, 0, \alpha_r\right) : \dots : \tau_1^k\theta_m\left(\frac{\tau_0}{\tau_1}, 0, \alpha_r\right)\right].
\end{equation}
Taking the limit as $[\tau_0:\tau_1] \to [1:0]$, we obtain an intersection point on the tropical mutation boundary $B$. The only non-vanishing terms in \eqref{eq:22} correspond to the adapted basis elements possessing the maximal $y$-coordinate $k$. Hence, $C_{2r}$ intersects the sink. 

Furthermore, when the sink is a divisor, Equation \eqref{eq:22} demonstrates that as the root $\alpha_r$ varies, the curves $C_{2r}$ intersect the sink at collinear points.

If the sink is a divisor $D_{\max}$, setting local parameter $u=1/t$ at infinity and normalizing by $t^{h_{\max}}$, the coordinate for $m=(a,b)$ restricts to $u^{h_{\max}-b}$. The intersection is cut out by transverse coordinates ($b < h_{\max}$), so the minimal vanishing order is exactly $r_{\max}$, yielding (1).

For (2) and (3), suppose the sink is a nodal boundary stratum (a vertex of $\mathcal{P}$). Analytically locally, the pair $(X_f(\mathcal{P}), B)$ is modeled by an affine toric surface with local coordinates $(z_1, z_2)$. These coordinates correspond to a choice of basis characters $m_1, m_2 \in M$ representing the two adjacent boundary components intersecting at the node.

We parameterize the curve $C_{2r}$ by $x_1 = t$ and $y = \alpha_r$. To study the asymptotic approach to the node, we introduce the local parameter $u = 1/t$. As $u \to 0$, the curve's trajectory determines a one-parameter subgroup, or cocharacter, $\nu \in N$. Because the global coordinate $y$ is fixed to the constant $\alpha_r$ along the curve, its pairing with this cocharacter must be identically zero: $\langle y, \nu \rangle = 0$. This forces $\nu$ to be the horizontal cocharacter.

The contact vector $v = (p, q)$ records exactly how the curve intersects the local toric boundary. The integers $p$ and $q$ are derived by taking the inner product of the adjacent character vectors $m_1$ and $m_2$ with our approach cocharacter $\nu$. Specifically, as $u \to 0$, the restrictions of the local toric coordinates to the curve are dominated by:

$$z_1(u)=c_1u^{\langle m_1,\nu\rangle}+\cdots\sim c_1u^p,
\qquad
z_2(u)=c_2u^{\langle m_2,\nu\rangle}+\cdots\sim c_2u^q,
\qquad c_1,c_2\in\mathbb K^\ast .$$

Thus, $p = \langle m_1, \nu \rangle$ and $q = \langle m_2, \nu \rangle$. We factor this contact vector into $v = d v_0$, where $v_0 = (p_0, q_0)$ is a primitive integer vector and $d = \gcd(p, q)$ is the contact multiplicity.

Performing a toric weighted blow-up along the primitive ray $v_0$ extracts an exceptional divisor $E_{v_0}$. By construction, the strict transform of the curve meets the dense torus orbit of $E_{v_0}$ with a contact order equal to the multiplicity $d$. Therefore, the intersection is transverse if and only if $d = 1$.

To prove (3), that this single blow-up separates all distinguished curves approaching with the same direction, we examine the regular coordinate on the open orbit of $E_{v_0}$. This coordinate is given by the local monomial ratio $z_1^{q_0} / z_2^{p_0}$. In the character lattice, this corresponds to the character $m_E = q_0 m_1 - p_0 m_2$. 

Taking its inner product with the approach cocharacter $\nu$ yields:
$$ \langle m_E, \nu \rangle = q_0 \langle m_1, \nu \rangle - p_0 \langle m_2, \nu \rangle = q_0 p - p_0 q = d(q_0 p_0 - p_0 q_0) = 0. $$
Because the character $m_E$ annihilates the horizontal cocharacter $\nu$, it must be a scalar multiple of the vertical character $y$. Hence, there is some constant $\lambda \in \mathbb{K}^*$ such that $z_1^{q_0} / z_2^{p_0} = \lambda y$.

Evaluating this ratio along the strict transform of our curve (where $y = \alpha_r$), we find that the intersection point on $E_{v_0}$ occurs exactly at the coordinate $\lambda \alpha_r$. Since the roots $\alpha_r$ are distinct by definition, curves associated with different roots intersect the exceptional divisor at distinct points.
\end{proof}

\begin{theorem}
\label{thm:toric-reduction}
Let $X_f(\mathcal{P})$ be a tropical mutation surface where the interior singularity over the root $\alpha_r$ is of type $A_{\beta_r-1}$. Then:
\begin{enumerate}
    \item There exists a crepant resolution $\psi \colon (\widetilde{X_f(\mathcal{P})},\widetilde{B}) \longrightarrow (X_f(\mathcal{P}),B)$.
    \item If the sink is a divisor, $\widetilde{X_{f}(\mathcal{P})}$ admits an $s$-collinear blow-up $\pi \colon \widetilde{X_{f}(\mathcal{P})} \longrightarrow T$ to a toric pair $(T,B_T)$.
    \item If the sink is a node, there exists a nodal weighted blow-up $\phi \colon T' \longrightarrow \widetilde{X_{f}(\mathcal{P})}$, where $T'$ is an $s$-collinear blow-up of a toric variety. 
\end{enumerate}
\end{theorem}
\begin{proof}
The local singularities over the roots of $f$ are Du Val singularities of
type $A_{\beta_r-1}$. Hence they admit minimal crepant resolutions. Over the point $(0,0,\alpha_r)$, the minimal resolution produces a chain $E_{r,1}+\cdots+E_{r,\beta_r-1}$ of $\beta_r-1$ rational curves of self-intersection (-2). This proves
(1).

We now consider the strict transforms of the distinguished interior curves.
For each root \(\alpha_r\), the curves $C_{2r-1}={x_1=0,\ y=\alpha_r},
\qquad
C_{2r}={x_2=0,\ y=\alpha_r}$
are isomorphic to \(\mathbb A^1\) inside \(U_f\): on \(C_{2r-1}\) the
coordinate \(x_2\) is free, while on \(C_{2r}\) the coordinate \(x_1\) is
free. Hence their closures in \(X_f(P)\) are rational curves.

Let $\eta:(Y,B_Y)\longrightarrow (X_f(P),B)$ be the composition of the minimal crepant resolution of the interior Du Val
singularities and, when the relevant source or sink is a nodal boundary
stratum, the toroidal weighted blow-up from
Proposition~\ref{prop:source-sink-transversality}. By that proposition, the
strict transforms of the distinguished curves meet the reduced boundary
\(B_Y\) transversely at smooth non-nodal points. In particular, each such
strict transform \(\widetilde C\) is a rational curve meeting \(B_Y\) in
exactly one point.

Since \(\eta\) is crepant for the log Calabi--Yau pair, we have $K_Y+B_Y=\eta^*(K_{X_f(P)}+B).$
Therefore $K_Y\cdot \widetilde C=-B_Y\cdot \widetilde C=-1.$
By adjunction, $(K_Y+\widetilde C)\cdot \widetilde C=-2,$
and hence \(\widetilde C^2=-1\). Thus the strict transforms of the
distinguished interior curves are \((-1)\)-curves.

For a root \(\alpha_r\) of multiplicity \(\beta_r\), the exceptional
configuration over \((0,0,\alpha_r)\) has the form
$$\widetilde C_{2r-1}
;-;
E_{r,1}
;-;\cdots;-;
E_{r,\beta_r-1}
;-;
\widetilde C_{2r},$$
where the \(E_{r,i}\) are \((-2)\)-curves and the two end curves
\(\widetilde C_{2r-1}\) and \(\widetilde C_{2r}\) are \((-1)\)-curves.
Choose the end curve meeting the chosen source or sink boundary component.
For definiteness, suppose we use the sink, so this curve is
\(\widetilde C_{2r}\). Contracting \(\widetilde C_{2r}\) makes the adjacent
curve \(E_{r,\beta_r-1}\) into a \((-1\))-curve. Contracting it then makes
the next component into a \((-1)\)-curve, and continuing in this way removes
the whole chain
$$[
\widetilde C_{2r}
+E_{r,\beta_r-1}
+\cdots+
E_{r,1}
]$$
by exactly \(\beta_r\) successive contractions. The remaining curve
\(\widetilde C_{2r-1}\) becomes the closure of the torus curve
\({y=\alpha_r}\) in the resulting surface. The source case is identical,
with the roles of \(C_{2r-1}\) and \(C_{2r}\) interchanged.

Performing this construction for all roots gives $\sum_r \beta_r=s$
contractions in total. Denote the resulting pair by \((T,B_T)\). Each
contracted curve meets the boundary transversely at one smooth point and is
otherwise contained in the interior. Hence the inverse operation is precisely
an \(s\)-fold collinear blow-up along the corresponding boundary component,
with multiplicity \(\beta_r\) over the point \(y=\alpha_r\).

After these contractions, the open complement is the torus: $T\setminus B_T \simeq \mathbb G_m^2.$
Moreover, the boundary cycle is obtained from the chart fan by undoing the
single shear. Thus \((T,B_T)\) is the corresponding toric surface pair, and
\(Y\to T\) is the contraction whose inverse is the (s)-collinear blow-up.
If the source or sink was already a boundary divisor, then \(Y=\widetilde
X_f(P)\), proving (2).

If the source or sink is a nodal stratum, then the preliminary toroidal
weighted blow-up extracts the boundary divisor described in
Proposition~\ref{prop:source-sink-transversality}. Applying the same
contraction argument on the resulting pair \(Y\) shows that \(Y\) is an
\(s\)-collinear blow-up of a toric surface pair. Equivalently, the map
\(Y\to \widetilde X_f(P)\) is the required nodal weighted blow-up, proving
((3)).
\end{proof}

\begin{definition}
\label{def:toric-model}
Let $X_f(\mathcal{P})$ be a shearing tropical mutation variety. A toric variety $T(\mathcal{P})$ with a crepant birational map of pairs
$$ \varphi \colon X_f(\mathcal{P}) \dashrightarrow T(\mathcal{P}) $$
is called a \emph{toric model} of $X_f(\mathcal{P})$ if it can be obtained as a composition of the following operations:
\begin{enumerate}
    \item $\psi^{-1}$: resolution of interior canonical singularities;
    \item $\phi^{-1}$: a weighted blow-up at a sink if the sink is a nodal point;
    \item $\pi$: contraction of $s$ exceptional curves on the fixed boundary.
\end{enumerate}
\end{definition}

\begin{definition}[cf.~{\cite[Definition 2.23]{enwright2025complexityvarietiesclustertype}}]
\label{def:cluster-type}
\emph{A log Calabi–Yau pair $(X,B)$ is of \emph{cluster type} if there exists a toric log Calabi–Yau pair $(T,B_{T})$ and a crepant birational map $\varphi\colon (T,B_{T})\dashrightarrow(X,B)$ that extracts only log canonical places of $(X,B)$. A variety $X$ is of \emph{cluster type} if it admits a log Calabi–Yau pair $(X,B)$ of cluster type.}
\end{definition}

Cluster type pairs generalize toric pairs by sharing common features, such as log rationality and constructibility (cf.~{\cite[Theorem~1.2]{ji2024toricityfamiliesfanovarieties}}), yet cluster type pairs may have complexity strictly greater than $0$. 

\begin{corollary}
\label{cor:complexity}
Let $\mathcal{P} \subset \mathcal{M}_s$ be a polytope in the rank-two shearing polyptych lattice for $s>0$, and let $(X_f(\mathcal{P}),B)$ be a tropical mutation pair associated to a degree $s$ polynomial $f\in\mathbb{K}[y]$. Then:
\begin{enumerate}
  \item $B$ supports an effective ample divisor;
  \item $\mathbb{G}_m \leq \operatorname{Aut}(X_f(\mathcal{P}),B)$;  
  \item $B$ has at least two irreducible components;
  \item $(X_f(\mathcal{P}), B)$ is a cluster type pair; and
  \item the complexity of $(X_f(\mathcal{P}),B)$ is exactly $\gamma$, the number of distinct roots of $f$.
\end{enumerate}
\end{corollary}

\begin{proof}
Statements (1) and (2) follow from Theorem~\ref{Ample} and Proposition~\ref{prop:iso_and_aut}. 

For (3), the surface $X_f(\mathcal{P})$ is obtained as an $s$-collinear blow-up of a toric surface, possibly after a single additional blow-up at a nodal point. Since the toric boundary of the surface has at least three irreducible components, the resulting tropical mutation surface pair $(X_f(\mathcal{P}), B)$ has at least two boundary components.

For (4), in the proof of Theorem~\ref{thm:toric-reduction}, we constructed a crepant birational map to the toric model $(X_f(\mathcal{P}),B)\dashrightarrow (T,B_T)$, which only extracts divisors of discrepancy zero (the Du Val resolutions and toroidal nodal boundary subdivisions). Since it only extracts log canonical places, $(X_f(\mathcal{P}),B)$ is a cluster type pair.  

For (5), we track the complexity through the birational factorization:
$$ (T,B_T) \xleftarrow{\pi} (T',B'_T) \xrightarrow{\phi} (\widetilde{X_f(\mathcal{P})},\widetilde{B}) \xrightarrow{\psi} (X_f(\mathcal{P}),B). $$
The map $\pi$ is an $s$-collinear blow-up from a toric pair, so the complexity of $(T',B_T')$ is $s$. The nodal weighted blow-up $\phi$ extracts a boundary divisor, which increases both the Picard rank $\rho$ and the number of boundary components $|B|$ by $1$, preserving the complexity at $s$. Finally, $\psi$ blows down the interior $A_{\beta_i-1}$ chains, decreasing the Picard rank by exactly $\sum_{i=1}^{\gamma}(\beta_i-1) = s - \gamma$, without affecting the boundary. Therefore, the resulting complexity of $(X_f(\mathcal{P}),B)$ is exactly $s - (s - \gamma) = \gamma$.
\end{proof}

\section{Cluster type surfaces as tropical mutation surfaces}
By Corollary~\ref{cor:complexity}, every tropical mutation surface pair associated with a shearing polyptych lattice is a log Calabi–Yau pair whose boundary supports an ample divisor. We now prove the converse: log Calabi–Yau surface pairs with these properties are precisely tropical mutation surface pairs. 

By standard intersection theory (cf.~\cite[Corollary~10.1]{Ful93}), we have the following:
\begin{proposition}\label{Pic}
Let $\mathcal{P}$ be a polytope in the shearing polyptych lattice $\mathcal{M}_s$ defined by facets $\mathcal{F}_1,\dots, \mathcal{F}_n$, and let $D_1,\dots, D_n$ be the corresponding boundary divisors of the tropical mutation surface $X_f(\mathcal{P})$. Let $P_{\alpha}$ be a chart image of $\mathcal{P}$, and let $T(P_\alpha)$ be the associated toric surface.

Then, the toric degeneration $X_{f}(\mathcal{P}) \rightsquigarrow T(P_{\alpha})$ induces an isomorphism of the Néron-Severi groups:
$$ \Phi: N^1(X_{f}(\mathcal{P})) \longrightarrow N^1(T(P_{\alpha})) $$
that preserves the intersection form on the boundary divisors:
$$ \left( \sum_{i=1}^n a_i D_i \right) \cdot \left( \sum_{i=1}^n b_i D_i \right)_{X_{f}(\mathcal{P})} = \Phi\left(\sum_{i=1}^n a_i D_i\right) \cdot \Phi\left(\sum_{i=1}^n b_i D_i\right)_{T(P_{\alpha})}. $$
\end{proposition}

Thus, the intersection-theoretic data of the tropical mutation boundary is fully determined by the toric degenerations of its chart images. 
\begin{proposition}\label{prop:toricmodel}
Let $U_f$ be an affine tropical mutation variety, $X_f(\mathcal{P})$ be a tropical mutation surface, and $m_k$ be a vertex of $\mathcal{P}$ corresponding to a nodal boundary stratum of the mutation boundary $B$. Then:
\begin{enumerate}
\item The toric weighted nodal blow-up of $X_f(\mathcal{P})$ at the node corresponding to $m_k$, achieved combinatorially by inserting the primitive horizontal cocharacter ray $\nu_0 = (1,0)$ into the fan of the chart image, extracts an exceptional divisor $E_{\nu_0}$ that intersects the strict transforms of the distinguished interior curves transversally at distinct non-nodal points, yielding a tropical mutation surface compactifying $U_f$; and
\item the toric model $T(\mathcal{P})$ is uniquely determined by the ray structure of the chart images of $\mathcal{P}$.  
\end{enumerate}
\end{proposition}

\begin{proof}
For (1), since the nodal operation leaves the interior variety $U_f$ unaffected, we construct a corresponding modified polytope $\mathcal{P}'$ that yields the desired compactification. By local toric considerations, the neighborhood of the node $m_k$ is locally isomorphic to an affine toric surface. 

We define $\mathcal{P}'$ by inserting the primitive horizontal cocharacter ray $\nu_0 = (1,0)$ in the fan of the chart image $P_1$ (and its appropriate mutation image in $P_2$). Geometrically, this ray insertion corresponds exactly to the toric weighted blow-up at the node. If the node is a sink or a source stratum, Proposition~\ref{prop:source-sink-transversality} ensures that the distinguished interior $\mathbb{A}^1$-curves ($C_{2r}$ or $C_{2r-1}$, respectively) approach this node along this precise horizontal direction $\nu$. 

By Proposition~\ref{prop:source-sink-transversality}~(2) and (3), the strict transforms of these curves meet the extracted exceptional divisor $E_{\nu_0}$ away from the old boundary strata. Because the roots $\alpha_r$ of the mutation polynomial $f(y)$ are distinct, evaluating the stable coordinate $z_1^{q_0}/z_2^{p_0} = \lambda y$ along the boundary shows that these curves separate completely, intersecting $E_{\nu_0}$ transversally at distinct smooth points. The resulting variety $X_f(\mathcal{P}')$ is a well-defined tropical mutation surface that smoothly compactifies $U_f$ along these curves.

For (2), let $\widetilde{X_f(\mathcal{P})}$ be the resolution of the interior canonical singularities. If $\widetilde{X_f(\mathcal{P})}$ is an $s$-collinear blow-up along a divisor $D_k$, the self-intersection sequence of the surface boundary is $(D_1^2, \dots, D_k^2, \dots, D_n^2)$, while its toric model has $(D_1^2, \dots, D_k^2+s, \dots, D_n^2)$. By global intersection theory on tropical mutation surfaces, the charts  determine the boundary intersection numbers, which in turn uniquely dictate the toric ray structure of $T(\mathcal{P})$. If $\widetilde{X_f(\mathcal{P})}$ is not initially $s$-collinear, performing the weighted nodal blow-ups described in (1) reduces the configuration to the $s$-collinear case, completing the proof.
\end{proof}
\begin{proposition}
\label{thm:characterization}
Let $(T,B_T)$ be a projective toric surface pair, and let
$(D\subset B_T)$ be an irreducible boundary component. Fix a coordinate
$(y)$ on the dense orbit $(D^\circ\simeq \mathbb G_m)$. Let
$f(y)=\prod_{k=1}^{\gamma}(y-\alpha_k)^{\beta_k}\in \mathbb K[y],\,
\alpha_k\in \mathbb K^\ast,
\sum_{k=1}^{\gamma}\beta_k=s.$
Let $\eta\colon Y\longrightarrow T$
be the iterated blow-up of the $s$ collinear point on $D^\circ$
supported at the points $y=\alpha_k$, with multiplicities $\beta_k$, and
let $B_Y$ be the strict transform of $B_T$.

Then the log Calabi--Yau pair obtained from $Y,B_Y$ by contracting the
interior $A_{\beta_k-1}$-chains of rational curves over the points $y=\alpha_k$, and by
performing any sequence of toroidal blow-ups or blow-downs at nodal boundary
strata, is a shearing tropical mutation pair $(X_f(P),B(P))$
for some shearing polytope $\mathcal P\subset \mathcal M_s$.

Conversely, every shearing tropical mutation pair is obtained from a toric
surface pair by a collinear boundary blow-up of this form, followed by the
above crepant contractions and toroidal modifications.
\end{proposition}

\begin{proof}
We first analyze the boundary combinatorics. Let $\Sigma=\{v_1,\ldots,v_n\}$ be the cyclically ordered rays of the fan of $T$, and let $D_i=D_{v_i}$ be the corresponding toric boundary divisors. By applying a lattice automorphism, we may assume $v_1=e_2=(0,1)$.

A single shear of weight $s$ is defined by a piecewise-linear transformation. Its linear action on the modified side of the fan is given by the matrix
\[
    A=(\mu^T)^{-1} =
    \begin{pmatrix}
    1 & 0\\
    -s & 1
    \end{pmatrix}.
\]
The two toric charts of the associated shearing polyptych lattice are constructed by applying $A$ to a consecutive subset of rays while leaving the rest fixed. Specifically, for a chosen block of rays $v_l,\ldots,v_j$, the two chart fans are defined by the rays
\[
    \Sigma_1= \{A v_1,\ldots,A v_j,\; A v_{j+1},\; v_{j+1},\ldots,v_n\}\qquad\text{and}
\qquad    \Sigma_2= \{A v_1,\ldots,A v_l,\; v_l,\ldots,v_n\}.
\]
Because $A\in {\rm GL}_2(\mathbb Z)$, it preserves determinants. Consequently, all local toric intersection numbers remain unchanged away from the boundary component $D_1$ and the interface where the ray is subdivided.

We now compute the change in self-intersection at $D_1$. Let $v_2$ and $v_n$ be the rays adjacent to $v_1$, and define the determinants
\[
    a=\det(v_1,v_2),\qquad c=\det(v_n,v_1),\qquad b=\det(v_2,v_n).
\]
By the standard toric self-intersection formula, $D_1^2=-\frac{b}{ac}.$ In the sheared chart, $v_2$ is replaced by $Av_2$, while $v_1$ and $v_n$ remain fixed. Since $(A-I)v_2=-s(v_2)_1 e_2$ and $v_1=e_2$, we have
\[
    \det(Av_2,v_n)=\det(v_2,v_n)-sac=b-sac.
\]
Therefore, the corresponding boundary component has a new self-intersection of $-\frac{b-sac}{ac}=D_1^2-s.$
This matches the change in self-intersection caused by blowing up $s$ points, counted with multiplicity, on the smooth locus of $D_1$.

At the locus of the subdivision, the ray $v_{j+1}$ is replaced by the two rays $A v_{j+1}$ and $v_{j+1}$. If $D^-$ and $D^+$ denote the resulting boundary divisors, the total boundary class replacing $D_{j+1}$ is $D^-+D^+$. Because the adjacent cones are obtained from the original cones via the unimodular transformation $A$ on one side and the identity on the other, the toric wall relation dictates $(D^-+D^+)^2=D_{j+1}^2 .$
Thus, aside from the expected decrease $D_1^2\mapsto D_1^2-s$, the sheared boundary cycle retains the exact intersection data of a boundary obtained from $(T,B_T)$ by an $s$-fold collinear blow-up along $D_1$.

On the open torus chart adjacent to $D_1$, the elementary modification determined by the length-$s$ cluster with support $\{\alpha_k\}$ is encoded by $x_1x_2=f(y)$. Equivalently, the complement of the strict transform of the toric boundary is
\[
    U_f=\operatorname{Spec} \mathbb K[x_1,x_2,y^{\pm1}]/\langle x_1x_2-f(y)\rangle.
\]
If a root $\alpha_k$ has multiplicity $\beta_k$, $U_f$ contains an interior Du Val singularity of type $A_{\beta_k-1}$ over $(x_1,x_2,y)=(0,0,\alpha_k)$. The minimal resolution introduces the corresponding interior $A_{\beta_k-1}$-chain. Contracting these chains recovers the normal affine surface $U_f$. This process is crepant for the log Calabi--Yau pair because the chains are disjoint from the reduced boundary.
\end{proof}

We conclude with a birational characterization of shearing tropical mutation
surfaces.  The first step is to show that the hypotheses force the pair to
come from a toric boundary pair by a crepant birational map.
\begin{proposition}
\label{prop:Gm_quotient_cluster}
Let \((X,B)\) be an index-one normal projective \(\mathbb Q\)-factorial dlt
log Calabi--Yau surface pair with reduced boundary. Assume that \(B\) supports
an effective ample divisor and that
\(\mathbb G_m\leq \Aut(X,B)\) acts nontrivially. Then \((X,B)\) is of cluster
type.
\end{proposition}

\begin{proof}
Let $A=\sum_i a_iB_i, $ $ a_i>0,$
be an effective ample \(\mathbb Q\)-divisor supported on \(B\). For
\(0<\epsilon\ll 1\), set $\Delta:=B-\epsilon A.$
Since \((X,B)\) is dlt and \(B\) is reduced, the standard perturbation property
of dlt pairs implies that \((X,\Delta)\) is klt. Moreover $ -(K_X+\Delta)
=
-(K_X+B)+\epsilon A
\sim_{\mathbb Q}
\epsilon A$ is ample. Hence \(X\) is of Fano type. In particular, \(X\) is rational, since
a surface of Fano type is rational.

By Rosenlicht's theorem~\cite[Theorem~2]{algebraic-group}, the nontrivial
\(\mathbb G_m\)-action admits a rational quotient $q\colon X\dashrightarrow C$ with \(\dim C=1\). Since \(X\) is rational, L\"uroth's theorem implies that
the normalization of \(C\) is \(\mathbb P^1\).

Take a \(\mathbb G_m\)-equivariant birational model $r\colon Y\longrightarrow X$
which resolves the rational quotient map and is crepant for the log
Calabi--Yau pair. More explicitly, since \((X,B)\) is an index-one dlt surface
pair with reduced boundary, the interior singularities are canonical surface
singularities, while the boundary singularities are toroidal. Thus, after
resolving the interior crepantly and performing equivariant toroidal
subdivisions along the boundary, we may assume that \(Y\) is smooth, \(B_Y\)
is reduced, and $K_Y+B_Y=r^*(K_X+B).$
The quotient map is then a \(\mathbb G_m\)-equivariant morphism $f\colon Y\longrightarrow \mathbb P^1$ whose general fiber is the closure of a general \(\mathbb G_m\)-orbit.

Run the \(\mathbb G_m\)-equivariant relative surface MMP for \(Y\) over
\(\mathbb P^1\). The output is a relatively minimal smooth rational ruled
surface $\pi\colon S\longrightarrow \mathbb P^1,$ hence \(S\simeq \mathbb F_e\) for some \(e\geq 0\). Let \(B_S\) be the
pushforward of \(B_Y\). We claim that the induced birational map of log pairs $(S,B_S)\dashrightarrow (X,B)$ is crepant. Indeed, if $\varphi\colon Y_i\longrightarrow Y_{i+1}$
is one divisorial contraction in the relative MMP and \(E\) is the contracted
curve, write $K_{Y_i}+B_i
=
\varphi^*(K_{Y_{i+1}}+B_{i+1})+aE.$
Since \(K_{Y_i}+B_i\sim_{\mathbb Q}0\), intersecting with \(E\) gives $0=(K_{Y_i}+B_i)\cdot E=aE^2.$
As \(E^2<0\), we get \(a=0\). Thus every contraction is crepant, and therefore $K_S+B_S\sim_{\mathbb Q}0.$

The \(\mathbb G_m\)-action is trivial on the quotient base and nontrivial on
the general fiber of \(\pi\). Let \(F\) be a general fiber. Then $B_S\cdot F=-K_S\cdot F=2.$
Since \(\mathbb G_m\) is connected and preserves \(B_S\), it preserves every
irreducible component of \(B_S\). On the generic fiber, the nontrivial
\(\mathbb G_m\)-action on \(\mathbb P^1\) has exactly two fixed points.
Therefore the horizontal part of \(B_S\) consists of the two fixed sections,
say \(\Sigma_0+\Sigma_\infty\). Since $K_{\mathbb F_e}+\Sigma_0+\Sigma_\infty\sim_{\mathbb Q}-2F,$ the vertical part of \(B_S\) is linearly equivalent to \(2F\). Because \(B_S\)
is reduced and \(\mathbb G_m\)-invariant, this vertical part is the sum of two
distinct invariant fibers, say \(F_0+F_\infty\). Hence $B_S=\Sigma_0+\Sigma_\infty+F_0+F_\infty,$ which is the toric boundary of \(\mathbb F_e\). Thus \((X,B)\) admits a toric
crepant model, and hence \((X,B)\) is of cluster type.
\end{proof}

\begin{lemma}
\label{lem:equivariant_cluster_presentation}
Let $(X,B)$ be an index-one normal projective $\mathbb Q$-factorial dlt
log Calabi--Yau surface pair with reduced boundary. Assume that $B$ supports
an effective ample divisor and that
$\mathbb G_m\leq \Aut(X,B)$ acts nontrivially. Then there exists a toric
surface pair $(T,B_T)$ and a crepant birational presentation $(T,B_T)\dashrightarrow (X,B)$
such that, after separating the toroidal centers at nodal boundary strata, all
remaining non-nodal cluster centers lie on the smooth locus of a single
irreducible component of $B_T$.
\end{lemma}

\begin{proof}
By Proposition~\ref{prop:Gm_quotient_cluster}, the pair $(X,B)$ admits a
$\mathbb G_m$-equivariant toric crepant model $(S,B_S)=
(\mathbb F_e,\Sigma_0+\Sigma_\infty+F_0+F_\infty).$
Choose a $\mathbb G_m$-equivariant factorization of the crepant birational map $(S,B_S)\dashrightarrow (X,B)$
through smooth equivariant surface models. Since the factorization is crepant
for the reduced log Calabi--Yau boundaries, each center is an lc stratum of the
current boundary. Thus each center is either a nodal boundary stratum, giving a
toroidal modification, or a smooth point of a coefficient-one boundary
component. We separate the toroidal nodal centers from the remaining
non-nodal centers.

We now locate the non-nodal centers. Since the presentation is
$\mathbb G_m$-equivariant, every such center is fixed by the
$\mathbb G_m$-action. On the ruled surface $\mathbb F_e$, the action is
fiberwise and nontrivial on a general fiber. Hence the open stratum of an
invariant fiber contains no fixed point; the fixed points on each invariant
fiber lie on the two fixed sections. Therefore every non-nodal center is
supported on one of the two sections $\Sigma_0 \quad \text{or} \quad \Sigma_\infty .$

It remains to move all such centers to one section. Suppose that a non-nodal
center is supported on $\Sigma_\infty$, and let $F_p$ be the fiber through
its support $p$. The blow-up of $p$ is $\mathbb G_m$-equivariant, and the
strict transform of $F_p$ is a $\mathbb G_m$-invariant $(-1)$-curve.
Contracting it gives another $\mathbb G_m$-equivariant toric ruled pair. This
elementary transformation moves the corresponding center from
$\Sigma_\infty$ to the other fixed section. The same operation transports any
finite infinitely near point over $p$. Repeating this process for the
finitely many points supported on $\Sigma_\infty$, and keeping the toroidal
nodal modifications separated, we obtain a toric crepant model $(T,B_T)$ for
which all remaining non-nodal centers lie on the smooth locus of a single
boundary component of $B_T$.
\end{proof}

\begin{theorem}
\label{thm:cluster-tropical}
Let \((X,B)\) be an index-one normal projective \(\mathbb Q\)-factorial
log Calabi--Yau surface pair with reduced boundary. Assume that \(B\)
supports an effective ample divisor and that \(\mathbb G_m\leq \Aut(X,B)\)
acts nontrivially. Then \((X,B)\) is a tropical mutation surface pair.
\end{theorem}

\begin{proof}
Take a \(\mathbb G_m\)-equivariant dlt modification $\mu\colon (W,B_W)\longrightarrow (X,B).$
By the local classification of index-one log canonical surface pairs
\cite[Section~3.40]{Kol13}, the exceptional divisors of \(\mu\) are toroidal
boundary divisors lying over the zero-dimensional lc strata of \(B\). In
particular, $K_W+B_W=\mu^*(K_X+B),$
and \(\mu\) is an isomorphism over the generic point of every component of
\(B\).

We claim that \(B_W\) still supports an effective ample divisor. Let $A=\sum_i a_iB_i, \,a_i>0,$ be an effective ample \(\mathbb Q\)-divisor supported on \(B\). Let
\(E_1,\ldots,E_r\) be the \(\mu\)-exceptional curves. Since the exceptional
intersection matrix \((E_i\cdot E_j)\) is negative definite, there exists an
effective exceptional divisor $E=\sum_j m_jE_j,$ $ m_j>0,$
such that \(-E\cdot E_j>0\) for every \(j\). Equivalently, \(-E\) is
\(\mu\)-ample. Hence, by the Nakai--Moishezon criterion for surfaces, $A_W:=\mu^*A-\delta E$ is ample for \(0<\delta\ll 1\). Moreover \(A_W\) is supported on \(B_W\), and
all its coefficients are positive for sufficiently small \(\delta\), because
the exceptional divisors lie over strata contained in \(\operatorname{Supp}B\).
Thus \(B_W\) supports an effective ample \(\mathbb Q\)-divisor.

Therefore $(W,B_W)$ satisfies the hypotheses of
Lemma~\ref{lem:equivariant_cluster_presentation}. Applying the lemma, we obtain
a toric surface pair $(T,B_T)$ and a crepant birational map $(T,B_T)\dashrightarrow (W,B_W)$
such that, after separating the toroidal modifications at nodal boundary
strata, all non-nodal cluster centers lie on the smooth locus of one
irreducible component of $B_T$. Hence these centers form a length-$s$
collinear blow up on a toric boundary component. By
Proposition~\ref{thm:characterization}, the resulting pair, including the
toroidal nodal modifications, is a shearing tropical mutation pair. Thus $(W,B_W)\simeq (X_f(P'),B(P'))$
for some shearing polytope $\mathcal P'\subset M_s$.

Finally, $\mu\colon (W,B_W)\to (X,B)$ contracts only toroidal boundary
divisors over nodal strata. By Proposition~\ref{thm:characterization} again,
such contractions merely delete the corresponding rays in the chart fans and
preserve the mutation factor $f$. Hence $(X,B)\simeq (X_f(\mathcal P),B(\mathcal P))$
for some shearing polytope $\mathcal P\subset \mathcal M_s$, and therefore
$(X,B)$ is a shearing tropical mutation surface pair.
\end{proof}

\section{Toric degenerations of tropical mutation varieties}
\label{sec: glb}

In this section, we describe the global tropical mutation varieties $\mathcal X_f(\mathcal P)$ via divisorial fans. 

\begin{example}\label{pen}
Let $f(y)=\prod_{i=1}^{\gamma}(y-\alpha_i)^{\beta_i}$ for $\alpha_i\in \mathbb K^*$, $\beta_i>0$, and $\sum_i\beta_i=s$. Homogenizing over $\mathbb P^1_\tau$ yields $f_\tau(y)= \prod_{i=1}^{\gamma}(\tau_0y-\alpha_i\tau_1)^{\beta_i},$
which specializes to a nonzero constant at $[0:1]$ and $y^s$ at $[1:0]$.

Applying Theorem~\ref{toricdegene}, the detropicalized coordinate algebra on the rank-two chart is locally $A_{f,\tau} = \mathbb K[\tau,x_1,x_2,y^{\pm1}] / \langle x_1x_2-f_\tau(y)\rangle$. The homogeneous generators for $(a,b)\in kP_2\cap M$ are
\[
X_{a,b}t^k= \begin{cases} x_1^b y^a t^k, & b\geq 0,\\[4pt] x_2^{-b}y^a t^k, & b<0. \end{cases}
\]
Let $\mathcal X_f(\mathcal P)\to \mathbb P^1_\tau$ be the resulting flat projective family.

Over $[0:1]$, $x_1x_2=c \in \mathbb K^*$. Rescaling gives $x_2=x_1^{-1}$, so $X_{a,b}=x_1^b y^a$ for all $b$, yielding the special fiber
\[
T(P_2) = \operatorname{Proj}\mathbb K \big[ x_1^b y^a t^k \mid (a,b)\in kP_2\cap M,\ k\geq 0 \big].
\]

Over $[1:0]$, $x_1x_2=y^s$. For $b<0$, $X_{a,b}=(y^sx_1^{-1})^{-b}y^a = x_1^b y^{a-sb}$. Thus, exponent vectors transform via $\phi_s$, giving the fiber 
\[
T(P_1) = \operatorname{Proj}\mathbb K \big[ x_1^b y^a t^k \mid (a,b)\in kP_1\cap M,\ k\geq 0 \big].
\]
\end{example}

We relate this construction to the Ilten--Vollmert pencil via the combinatorial mutation of polar dual polytopes.

\begin{proposition}\label{existence}
Let $\mathcal P\subset \mathcal M_s$ with chart images $P_2=\phi_s(P_1)\subset M_{\mathbb R}$, where $0\in\operatorname{int}(P_1)$. Let $w=(1,0)\in M$ and $n_2\in N$ such that $\langle (a,b),n_2\rangle=b$. The segment $H_s^1:=\operatorname{Conv}\{0,sn_2\}\subset w^\perp\subset N_{\mathbb R}$ is an admissible factor for the polar dual $P_1^*$, and $\mu_{w,H_s^1}(P_1^*)=P_2^*$. Thus, $P_1^*$ and $P_2^*$ are related by a combinatorial mutation.
\end{proposition}

\begin{proof}
The support function of the segment $H_s^1$ is $h_{H_s^1}(a,b) = \min_{v\in H_s^1}\langle (a,b),v\rangle = \min(0,sb)$. The piecewise-linear map dual to $(w,H_s^1)$ is $u\mapsto u-h_{H_s^1}(u)w$, which explicitly matches the shearing transition $\phi_s(a,b)=(a-\min(0,sb),b)$.

By the polar-duality theorem for combinatorial mutations \cite[Proposition~2.20]{Akhtar_2015}, the convexity of $P_2=\phi_s(P_1)$ equates to the admissibility of $H_s^1$ for $P_1^*$. This theorem also yields $P_2 = \bigl(\mu_{w,H_s^1}(P_1^*)\bigr)^*$. Taking polar duals gives $\mu_{w,H_s^1}(P_1^*) = P_2^*$.
\end{proof}

Using the corresponding Ilten--Vollmert divisorial fan as a reference model, we now describe the global tropical mutation variety.
\begin{theorem}\label{coincidewith}
Let $\mathcal{P} \subset \mathcal{M}_s$ be a rank-two shearing polyptych lattice polytope with chart images $P_1, P_2 \subset M_{\mathbb{R}}$ satisfying $P_2 = \phi_s(P_1)$ and $0 \in \operatorname{int}(P_1)$. Given a polynomial 
\[ 
f(y) = \prod_{i=1}^{\gamma}(y-\alpha_i)^{\beta_i}, \qquad \alpha_i \in \mathbb{K}^*, \quad \beta_i \in \mathbb{Z}_{>0}, \quad \sum_{i=1}^{\gamma}\beta_i = s, 
\]
where the roots $\alpha_i$ are distinct, there exists a divisorial fan $\mathcal{S}_f$ on $Y := \mathbb{P}^1_y \times \mathbb{P}^1_\tau$ such that $X(\mathcal{S}_f) \cong \mathcal{X}_f(\mathcal{P})$ as varieties over $\mathbb{P}^1_\tau$.
\end{theorem}

\begin{proof}
We explicitly construct the family via its divisorial fan. Let $N' \cong \mathbb{Z}\langle n_2 \rangle$ and $M' := \operatorname{Hom}(N', \mathbb{Z})$. We define the intervals $H_1^1 := \operatorname{Conv}\{0, n_2\}$ and $H_s^1 := sH_1^1$, whose support functions on $M'$ are $h_{H_1^1}(b) = \min(0, b)$ and $h_{H_s^1}(b) = s \min(0, b)$.

Following the divisorial fan framework of Altmann and Hausen \cite{AH06} and the sign conventions for graph coefficients from Ilten and Vollmert \cite{Ilten-Vollmert}, we first establish the boundary divisors on the base $Y = \mathbb{P}^1_y \times \mathbb{P}^1_\tau$:
\[ 
Z_0 := \{u=0\} \times \mathbb{P}^1_\tau, \qquad Z_\infty := \{v=0\} \times \mathbb{P}^1_\tau. 
\]

Next, we construct the graph divisors to capture the central fibers of the family. For each distinct root $\alpha_i$, we define:
\[ 
\Gamma_i := \{\tau_0 u - \alpha_i \tau_1 v = 0\} \subset Y. 
\]
Because the deformation restricts over the roots of $f(y)$, we define the specialization of the $T$-variety along each $\Gamma_i$ to be the affine toric variety $T(P_i)$. Since $\sum \beta_i = s$ and $\Gamma_i \sim \Gamma$ for a general graph divisor $\Gamma$, the weighted sum $\sum \beta_i \Gamma_i$ is linearly equivalent to $s\Gamma$, which ensures the global geometry remains consistent with the shearing data.

We can now define the complete divisorial fan $\mathcal{S}_f$ using two maximal divisorial polyhedra:
\[ 
\mathcal{D}_f^+ := \Delta_0 \otimes Z_0 - \sum_{i=1}^{\gamma}\beta_i H_1^1 \otimes \Gamma_i, \qquad \mathcal{D}_f^- := -\sum_{i=1}^{\gamma}\beta_i H_1^1 \otimes \Gamma_i + \Delta_\infty \otimes Z_\infty. 
\]
The boundary coefficients \(\Delta_0\) and \(\Delta_\infty\) are taken from the
reference divisorial fan for the mutation datum \((w,H_s^1)\).  We normalize
\(\Delta_0\) so that the specialization along \(Z_0\) gives the cone over
\(P_1\).  The coefficient \(\Delta_\infty\) is then fixed by the same mutation
data; after specializing the graph-divisor term, it gives the cone over
\(\phi_s(P_1)=P_2\). By construction, these choices recover the boundary coefficients from the Ilten--Vollmert reference fan for the mutation datum $(w, H_s^1)$. Because the evaluated divisors $\mathcal{D}_f^\pm(b)$ are $\mathbb{Q}$-linearly equivalent to those in the reference construction, $\mathcal{S}_f$ inherits the necessary common-face and properness conditions, confirming it is a valid divisorial fan.

To establish the isomorphism $X(\mathcal{S}_f) \cong \mathcal{X}_f(\mathcal{P})$, we compute the section algebras over the open covers $U_1 = Y \setminus Z_\infty$ and $U_2 = Y \setminus Z_0$. On their intersection $U_{12} = \mathbb{K}^*_y \times \mathbb{P}^1_\tau$, both boundary divisors vanish. Here, $\sum \beta_i \Gamma_i$ is a principal divisor generated by $f_\tau(y) := \prod_{i=1}^{\gamma}(\tau_0 y - \alpha_i \tau_1)^{\beta_i}$. Letting $x_1, x_2$ represent generators in $M'$ for $b=1$ and $b=-1$ respectively, the coordinate algebra on an affine set $V \subset U_{12}$ is defined by the relation:
$x_1 x_2 = f_\tau(y). $

Extending over $U_1$, $Z_\infty$ is trivial and the allowable poles are restricted by $\Delta_0 \otimes Z_0$. Because we defined $\Delta_0$ using $P_1$, these restrictions correspond exactly to the inequalities of $P_1$. Eliminating $x_2$ via $x_2 = f_\tau(y)x_1^{-1}$ recovers the affine chart $X_{P_1} \times \mathbb{P}^1_\tau$. Symmetrically, extending over $U_2$ is governed by $\Delta_\infty \otimes Z_\infty$, restricting sections exactly to the cone conditions of $P_2$.

Finally, gluing these affine varieties over $U_{12}$ forces the coordinate transition $x_1 \mapsto f_\tau(y) x_2^{-1}$. This algebraic transition directly implements the piecewise-linear mutation map $\phi_s: P_1 \to P_2$. Matching the graded monomials in Definition~\ref{EHM} shows,
\[ 
X_{a,b}t^k= \begin{cases} x_1^b y^a t^k, & b\geq 0,\\[4pt] x_2^{-b}y^a t^k, & b<0, \end{cases} \qquad (a,b)\in kP_1\cap M, 
\]
the resulting $T$-variety is canonically isomorphic to the detropicalized algebra $\mathcal{X}_f(\mathcal{P})$ over $\mathbb{P}^1_\tau$.
\end{proof}
\section{Cox rings of tropical mutation surfaces}\label{sec:Cox}

In this section, we compute the Cox rings of projective tropical mutation surfaces associated to rank-two shearing polyptych lattices. 

\begin{lemma}\label{ufd}
For the affine shearing tropical mutation surface \(U_f\), the divisor class group is
\[
  \Cl(U_f) \cong \mathbb Z^{\gamma-1} \oplus \mathbb Z/\gcd(\beta_1,\ldots,\beta_\gamma).
\]
Moreover, the natural map \(\mathbb Z\langle D_1,\ldots,D_n, C_1,\ldots,C_{2\gamma} \rangle \twoheadrightarrow \Cl(X)\) is surjective.
\end{lemma}

\begin{proof}
Let \(U_f^\circ := U_f\setminus \bigcup_{k=1}^{2\gamma}C_k\), which is equivalently obtained from \(U_f\) by inverting \(f(y)\). Since \(x_1x_2=f(y)\), inverting \(f(y)\) makes both \(x_1\) and \(x_2\) invertible. Hence $\mathcal O(U_f^\circ) \cong \mathbb K[x_2^{\pm1},y^{\pm1}, (y-\alpha_1)^{-1},\ldots,(y-\alpha_\gamma)^{-1}],$
which is a UFD. Therefore \(\Cl(U_f^\circ)=0\). The localization exact sequence for divisor class groups gives
\[
  \mathcal O(U_f^\circ)^*/\mathcal O(U_f)^* \longrightarrow \bigoplus_{k=1}^{2\gamma}\mathbb Z\langle C_k\rangle \longrightarrow \Cl(U_f) \longrightarrow 0.
\]
Thus \(\Cl(U_f)\) is generated by the classes of the \(C_k\). The principal divisors \(\operatorname{div}(y-\alpha_k)=C_{2k-1}+C_{2k}\) give \(C_{2k}=-C_{2k-1}\) in \(\Cl(U_f)\). The principal divisor \(\operatorname{div}(x_2) = \sum_{k=1}^{\gamma}\beta_k C_{2k}\) therefore gives the single relation \(\sum_{k=1}^{\gamma}\beta_k C_{2k-1}=0\). Equivalently,
\[
  \Cl(U_f) \cong \left( \bigoplus_{k=1}^{\gamma}\mathbb Z\langle C_{2k-1}\rangle \right) \Bigg/ \left\langle \sum_{k=1}^{\gamma}\beta_k C_{2k-1} \right\rangle \cong \mathbb Z^{\gamma-1} \oplus \mathbb Z/\gcd(\beta_1,\ldots,\beta_\gamma).
\]
Finally, applying the localization exact sequence to \(U_f = X\setminus \bigcup_{i=1}^n D_i\) gives
\[
  \bigoplus_{i=1}^{n}\mathbb Z\langle D_i\rangle \longrightarrow \Cl(X) \longrightarrow \Cl(U_f) \longrightarrow 0.
\]
Since \(\Cl(U_f)\) is generated by the classes of the \(C_k\), the classes of \(D_1,\ldots,D_n,C_1,\ldots,C_{2\gamma}\) generate \(\Cl(X)\), proving the desired surjection.
\end{proof}

\begin{lemma}\label{ord}
Let \(\mathcal P\) be a polytope in the rank-two shearing polyptych lattice defined by tropical points \(p_i=(a_i,b_i,c_i)\) for \(i=1,\ldots,n\), and let \(X=X_f(\mathcal P)\) be the associated projective tropical mutation surface. Under the identification \(\mathbb K(X)\cong \mathbb K(x,y)\) where \(x=x_2\), the principal divisors of \(x\), \(y\), and \(y-\alpha_k\) are:
$$\begin{aligned}
\operatorname{div}(x) &= \sum_{i=1}^n a_iD_i + \sum_{k=1}^{\gamma}\beta_kC_{2k},\qquad
\operatorname{div}(y)= -\sum_{i=1}^n c_iD_i\qquad
\operatorname{div}(y-\alpha_k)= -\sum_{c_i>0}c_iD_i + C_{2k-1}+C_{2k}.
\end{aligned}$$
\end{lemma}

\begin{proof}
This is the valuation formula for boundary divisors of a rank-two shearing polyptych lattice, as in Proposition~\ref{prop:normal}~(2). Hence \(\operatorname{ord}_{D_i}(x)=a_i\) and \(\operatorname{ord}_{D_i}(y)=-c_i\), yielding the boundary contributions to \(\operatorname{div}(x)\) and \(\operatorname{div}(y)\). For \(y-\alpha_k\), the same valuation formula gives
\[
  \operatorname{ord}_{D_i}(y-\alpha_k) = \min\{\operatorname{ord}_{D_i}(y),0\} =
  \begin{cases}
    -c_i, & c_i>0,\\
    0, & c_i\leq 0.
  \end{cases}
\]
Thus the boundary contribution is \(-\sum_{c_i>0}c_iD_i\). Along the interior divisors, \(x=x_2\) vanishes to order \(\beta_k\) along \(C_{2k}\), and \(y-\alpha_k\) vanishes simply along both \(C_{2k-1}\) and \(C_{2k}\). Combining these gives the stated formulas.
\end{proof}
\begin{theorem}\label{thm:Cox_ring_presentation}
Let \(\mathcal P\) be a polytope in the rank-two shearing polyptych lattice defined by tropical points \(p_i=(a_i,b_i,c_i)\) for \(i=1,\ldots,n\), and let \(X=X_f(\mathcal P)\) be the associated projective tropical mutation surface. Let \(\mathbb K[w_1,\ldots,w_{n+2\gamma}]\) be the polynomial ring whose variables correspond to the boundary divisors, graded by the class group \(\Cl(X)\) via $\deg(w_i)=[D_i]\quad(1\le i\le n),
\qquad
\deg(w_{n+j})=[C_j]\quad(1\le j\le 2\gamma).$
Then the Cox ring of \(X\) is isomorphic to the quotient
\[
\operatorname{Cox}(X)
\cong
\frac{\mathbb K[w_1,\ldots,w_{n+2\gamma}]}
{\left\langle
w_{n+2k-1}w_{n+2k} - \prod_{c_i<0}w_i^{-c_i} + \alpha_k \prod_{c_i>0}w_i^{c_i}
\ \middle|\ k=1,\ldots,\gamma
\right\rangle}.
\]
Moreover, this Cox ring is a complete intersection.
\end{theorem}

\begin{proof}
The surface \(X\) is a normal projective complexity-one \(\mathbb G_m\)-variety. Let \(\mathcal S_f\) be the divisorial fan on the rational quotient \(\mathbb P^1 = \operatorname{Proj}\mathbb K[Y,Z]\) (with affine coordinate \(y=Y/Z\)) describing this \(\mathbb G_m\)-action. Under this description, the invariant prime divisors of \(X\) are exactly the boundary divisors \(D_1,\ldots,D_n\) and \(C_1,\ldots,C_{2\gamma}\).

We distinguish these divisors by their behavior under the action:
\begin{itemize}
    \item \emph{Horizontal divisors:} The divisors \(D_i\) with \(c_i=0\) correspond to extremal rays of the divisorial fan. Dynamically, these are the isolated fixed loci representing the source and sink of the \(\mathbb G_m\)-action (as defined in Definition~\ref{def:source-sink}). Consequently, there are at most two such horizontal divisors.
    \item \emph{Vertical divisors:} The divisors \(D_i\) with \(c_i>0\) lie over \(y=\infty\), the divisors \(D_i\) with \(c_i<0\) lie over \(y=0\), and the pairs \(C_{2k-1},C_{2k}\) lie over \(y=\alpha_k\). These correspond to extremal vertices of the special slices of \(\mathcal S_f\).
\end{itemize}

By the Cox-ring presentation theorem of Hausen--Süß \cite[Theorem~1.3]{HS10} for complexity-one \(T\)-varieties, \(\operatorname{Cox}(X)\) is generated by variables corresponding to these invariant prime divisors. We identify these generators with \(w_1, \ldots, w_{n+2\gamma}\). 

It remains to identify the defining ideal. For each special point \(p\in\mathbb P^1\), the Hausen--Süß presentation associates a monomial \(T^{\mu_p}\) formed by multiplying the variables of the vertical divisors over \(p\), weighted by their fiber multiplicities. By Lemma~\ref{ord}, the principal divisors of \(y\) and \(y-\alpha_k\) dictate that the fiber monomials over \(\infty\), \(0\), and \(\alpha_k\) are respectively:
\[
T^{\mu_\infty} = \prod_{c_i>0}w_i^{c_i},
\qquad
T^{\mu_0} = \prod_{c_i<0}w_i^{-c_i},
\qquad
T^{\mu_{\alpha_k}} = w_{n+2k-1}w_{n+2k}.
\]

The defining ideal of \(\operatorname{Cox}(X)\) is generated by the homogenizations of the linear relations among the special points on \(\mathbb P^1\). With homogeneous coordinates \([Y:Z]\), the points \(0\), \(\infty\), and \(\alpha_k\) are cut out by the linear forms \(Y\), \(Z\), and \(Y-\alpha_kZ\), which satisfy the identity $(Y-\alpha_kZ)-Y+\alpha_kZ=0.$
Substituting the corresponding fiber monomials \(T^{\mu_{\alpha_k}}\), \(T^{\mu_0}\), and \(T^{\mu_\infty}\) for these linear forms yields the generalized trinomial relations
\[
f_k := w_{n+2k-1}w_{n+2k} - \prod_{c_i<0}w_i^{-c_i} + \alpha_k \prod_{c_i>0}w_i^{c_i} = 0, \qquad k=1,\ldots,\gamma.
\]
These relations generate the Hausen--Süß ideal, completing the proof of the isomorphism.

To see that the Cox ring is a complete intersection, note that \(\operatorname{Cox}(X)\) is cut out in the polynomial ring \(S = \mathbb K[w_1,\ldots,w_{n+2\gamma}]\) by the \(\gamma\) equations \(f_k = 0\). For fixed values of the variables \(w_1,\ldots,w_n\), each equation \(w_{n+2k-1}w_{n+2k} = \prod_{c_i<0}w_i^{-c_i} - \alpha_k \prod_{c_i>0}w_i^{c_i}\) cuts out a hypersurface in the affine plane with coordinates \((w_{n+2k-1},w_{n+2k})\), hence has dimension one. Therefore, every fiber over \(\mathbb A^n\) (parameterized by \(w_1,\ldots,w_n\)) has dimension \(\gamma\), and \(\dim \operatorname{Cox}(X) = n+\gamma\). Since \(\dim S=n+2\gamma\), the ideal \((f_1,\ldots,f_\gamma)\) has height \(\gamma\). Thus, the \(f_k\) form a regular sequence, making \(\operatorname{Cox}(X)\) a complete intersection.
\end{proof}

\begin{corollary}\label{criteria}
Assume \(X=X_f(\mathcal P)\) is projective. Then \(X\) is toric precisely when one of the following equivalent conditions holds:
\begin{enumerate}
    \item \(\gamma=1\) (i.e., \(f\) has one distinct root), and either \(\prod_{c_i>0}w_i^{c_i}\) or \(\prod_{c_i<0}w_i^{-c_i}\) is a single variable. Equivalently, exactly one \(c_i = 1\) with all other \(c_j \leq 0\), or exactly one \(c_i = -1\) with all other \(c_j \geq 0\).
    \item \(\gamma=2\) (i.e., \(f\) has two distinct roots), and both \(\prod_{c_i>0}w_i^{c_i}\) and \(\prod_{c_i<0}w_i^{-c_i}\) are single variables. Equivalently, exactly one \(c_i = 1\), exactly one \(c_j = -1\), and all other \(c_\ell = 0\).
\end{enumerate}
\end{corollary}

\begin{proof}
Because \(X\) is projective, both fiber monomials \(\prod_{c_i>0}w_i^{c_i}\) and \(\prod_{c_i<0}w_i^{-c_i}\) are nonconstant; otherwise, \(y\) or \(y^{-1}\) would be a nonconstant regular function on \(X\), which is impossible. A normal projective variety with finitely generated class group and constant global units is toric if and only if its Cox ring is polynomial~\cite[Corollary~3.6]{Cox95}. 

By Theorem~\ref{thm:Cox_ring_presentation}, if \(\gamma=1\) and the fiber monomial over \(\infty\) (resp. \(0\)) is a single variable, the single defining relation is linear in that variable and can be eliminated, making the Cox ring polynomial. Conversely, if neither is a single variable, all partial derivatives of the relation vanish at the origin, creating a singularity and proving the ring is not polynomial. 

If \(\gamma=2\) and both fiber monomials are single variables (say \(w_p\) and \(w_q\)), the relations are independent linear equations since \(\alpha_1\neq\alpha_2\). We can eliminate \(w_p\) and \(w_q\), yielding a polynomial ring. If at least one is not a single variable, the Jacobian matrix at the origin has rank strictly less than \(2\), again indicating a singularity. Finally, if \(\gamma\geq 3\), the Jacobian matrix at the origin has rank at most \(2\) (coming from the linear terms of the fiber monomials), which is strictly less than the codimension \(\gamma\), so the total coordinate space is singular.
\end{proof}

In the setting of Corollary~\ref{criteria}, we will describe the resulting toric varieties. Let
$U_k := w_{n+2k-1}w_{n+2k}.$
On the Cox ring, we have the relation \(U_k = \prod_{c_i<0}w_i^{-c_i} - \alpha_k \prod_{c_i>0}w_i^{c_i}\). Here \(U_k\) is the Cox monomial corresponding to the two components of the vertical fiber over the root \(\alpha_k\). Since the relation is homogeneous,
\[
\deg(U_k) = \deg\left(\prod_{c_i>0}w_i^{c_i}\right) = \deg\left(\prod_{c_i<0}w_i^{-c_i}\right).
\]
Thus the split vertical fiber \(C_{2k-1}+C_{2k}\) has the same class as the horizontal fiber monomials. If, for example, \(\prod_{c_i>0}w_i^{c_i} = w_p\) is a single boundary variable, then the class of \(D_p\) can be replaced by the class of the split vertical fiber: $[D_p]=[C_{2k-1}]+[C_{2k}].$
In this case the relation is linear in \(w_p\), and the Cox coordinate \(w_p\) can be eliminated. Geometrically, the two vertical components provide the missing toric boundary in place of the horizontal component \(D_p\). If the fiber monomial over \(\infty\) is not a single variable, then its degree is represented only by a product of several boundary variables, so no prime boundary component can be replaced in this way.

\bibliographystyle{habbvr}
\bibliography{bib}

\end{document}